\documentclass[12pt]{amsart}
\usepackage{amssymb}
\usepackage{url}
\usepackage{epsfig}  	

\newtheorem{thm}{Theorem}

\newcommand{\dia}{$\diamondsuit$ }
\newcommand{\diaa}{$\diamondsuit\!$}
\newcommand{\diaas}{$\diamondsuit$\hskip -0.3 pt s }
\newcommand{\diab}{$\blacklozenge$ }
\newcommand{\diabb}{$\blacklozenge\!$}
\newcommand{\diabs}{$\blacklozenge$\hskip -0.3 pt s }
\newcommand{\whbox}{$\square$ }
\newcommand{\blbox}{$\blacksquare$ }
\newcommand{\whboxx}{$\square$\hskip -0.2 pt}
\newcommand{\blboxx}{$\blacksquare$\hskip -0.1 pt}

\begin{document}

\title[Colouring Isonemal Fabrics]{Colouring Isonemal Fabrics\\ with more than two Colours\\ and Non-twilly Redundancy}
\date{}
\thanks{Work on this material was done at home and at Wolfson College, Oxford. I am grateful for Oxford hospitality and Winnipeg patience. I am particularly grateful to photographer Allen Patterson for producing Figures \ref{fig:colCubeFront} and \ref{fig:colCubeBack} from my models.}

\author{R.S.D.~Thomas}
\address{St John's College and Department of Mathematics, University of Manitoba,  Winnipeg, Manitoba  R3T 2N2  Canada. orcid.org/0000-0003-4697-4209}
\email{robert.thomas@umanitoba.ca}

\begin{abstract}
Perfect colouring of isonemal fabrics by thin and thick striping of warp and weft with more than two colours is examined where the cells with warps and wefts of the same colour do not appear along diagonal lines (not twilly redundancy). In species 33 to 39, all fabrics of specific orders (linear periods) can be perfectly coloured by thin striping with satin redundancy. In fewer species, all fabrics of specific orders can be perfectly coloured by thick striping in two different ways with doubled satin redundancy. The isonemal fabric 6-1-1 can be used as a redundancy configuration. All these colourings allow the coloured weaving of flat tori and -- a few of them -- cubes.
\end{abstract}
\maketitle
\section{Introduction}
\label{sect:intro}

\noindent This paper presents a mathematical solution to a problem that is only partly mathematical. It lies in the mathematical study of weaving that was begun by Gr\"unbaum and Shephard about 35 years ago \cite{GS1980}. Their crucial insight was that weaving fabrics was made both more attractively symmetrical and mathematically interesting by demanding that the fabric be what they called {\em isonemal}, meaning that the symmetry group of the fabric is transitive on the strands composing it. The topology of a weave is {\em normally} represented by colouring vertical strands dark and horizontal strands pale so that each square {\em cell} of an array shows which strand goes over the other strand. (See Figures~\ref{fig:66abc}a, \ref{fig:81a73b}a). The arrays represent infinite periodic patterns. The attraction of the study has been mainly the visual appreciation of these arrays called {\em designs} with normal colouring. They have been catalogued%
\footnote{The catalogue numbers, which will identify many of the fabrics in the paper, were invented by Gr\"unbaum and Shephard for their first catalogue. There are three numbers if the third is known. The first is the common period along the strands of the fabric, called by them {\em order}. The second is `the decimal equivalent of the [order]-digit binary number obtained by interpreting each [dark cell] as 1 and each [pale cell] as 0. For an equivalence class of sequences, the {\em binary index} is defined as the minimum of the binary indices of all the sequences in the class.' \cite{GS1985} The third is an arbitrarily assigned serial number of fabrics with the same binary index if all are known. It is by no means always 1; cf.~Figure~\ref{fig:81a73b}.} 
(\cite{GS1985}, \cite{GS1986}, and \cite{HT1991}) and some counted \cite{GS1985}. In 1989 I decided that it would be interesting to investigate colouring the strands of isonemal fabrics periodically but not normally so that the symmetries of the fabric's design are {\em colour symmetries} of the coloured result, a {\em pattern}. This is called {\em perfect} colouring. I learned that this was only possible if parallel colours were arranged in stripes periodically, colour${}_1$, colour${}_2$, \dots , colour${}_p$, colour${}_1$, \dots \cite[Thm 1]{HT1991}, but with a co-author was able to do little more even for two colours in ignorance of the possible symmetry groups.
\begin{figure}
\[
\begin{array}{ccc}
\epsffile{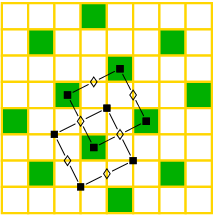} &\epsffile{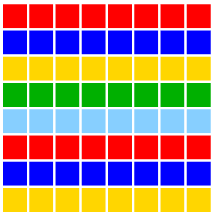} &\epsffile{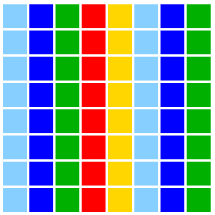}\\

\mbox{(a)} &\mbox{(b)} &\mbox{(c)}
\end{array}
\]
\caption{a. Normal colouring of the (5,3) satin 5-1-1 \cite[Figure 3a]{Thomas2013} with alternative $G_1$ lattice units. b. Five-colouring by thin striping. Obverse. c. Reverse. The black marks are explained in Section \ref{sect:sym}.}
\label{fig:66abc}
\end{figure}

Taking up the symmetry-group challenge, Richard Roth did a full determination of the symmetry groups of isonemal designs \cite{Roth1993} and also wrote a paper \cite{Roth1995} on two-colouring, indicating which symmetry groups went with perfect two-colorability. This allowed me to return to my quest for the interestingly symmetrical colorful weaving patterns that are displayed in this paper. There have been three stages. First I converted Roth's classification of 39 symmetry-group types, which necessarily overlap, to a taxonomy of 39 corresponding fabric species (with refinements), which do not overlap \cite{Thomas2009, Thomas2010a, Thomas2010b}.\begin{figure}
\[
\begin{array}{cc}
\epsffile{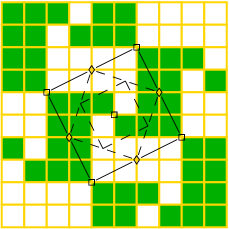} &\epsffile{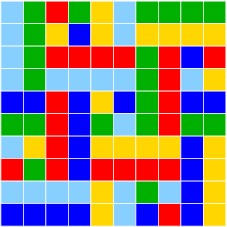}\\

\mbox{(a)} &\mbox{(b)}
\end{array}
\]
\caption{a. Normal colouring of Roth's example of species $33_3$, 10-55-2, with lattice unit of $G_1=H_1$ outlined and lattice units of levels 1 and 2 dashed \cite[Figure 11b]{Thomas2010b}. b. Five-colouring by thin striping.}
\label{fig:81a73b}
\end{figure}
Then I studied two-colouring \cite{Thomas2011, Thomas2012}. Then I studied the less interestingly symmetric multi-coloured patterns \cite{Thomas2013, Thomas2014}. This paper concludes the project with a study of what I regard as the most interestingly symmetric patterns. The respect in which I find these patterns more attractively symmetric is their quarter-turn symmetry, which allows more interesting (to me) motifs than half-turn or no-turn symmetries. Some designs with quarter-turn symmetry have the additional feature that they can be used to weave cubes --- not my idea \cite{HP2010,Pedersen1981,Pedersen1983,Thomas2010b,Thomas2012}, a bonus that I exemplify (Figures~\ref{fig:colCubeFront}, \ref{fig:colCubeBack}) but have not pursued even to the point of proposing a definition of what an isonemal cube should be.

The organization of the paper is as follows. Section~\ref{sect:prelim} introduces notation. Section~\ref{sect:sym} discusses symmetry. Section~\ref{sect:isonemal} discusses isonemal fabrics. Section~\ref{sect:col} discusses colouring other than normal. Section~\ref{sect:results} presents the main results on 5-colouring with (5,3)-satin redundancy. Section~\ref{sect:larger} extends the results of Section~\ref{sect:results} to higher orders. Section~\ref{sect:nonsatin} introduces a non-twilly redundancy configuration for 6-colouring. Section~\ref{sect:nonplanar} uses the weaving designs of Sections~\ref{sect:results} and \ref{sect:nonsatin} to weave closed surfaces.

\section{Preliminaries}
\label{sect:prelim}

\noindent
Mathematical weaving concerns the properties of certain diagrams that are meant to represent finite physical objects that are, loosely speaking, woven. This includes both cloth-like material made of fibres roughly circular in cross section and caning or basketry made of material wider than thick. A convention is that the {\em strands} that are woven are edgeless planar stips of uniform width and zero thickness of the same colour on both sides. While weaves more than 2 strands thick and with non-perpendicular strands have been studied \cite{GS1988}, this paper concerns perpendicular strands almost covering the plane twice, interleaved so that one strand covers the other in what I refer to as square {\em cells}. The idealized visual appearance of the weave therefore consists of a finite matrix of coloured cells intended to represent an infinite array following the pattern indicated by the diagram. (See Figure~\ref{fig:66abc}a, which has features beyond just its array of cells and whose caption uses terms that will be defined later. This is by no means the only reference to this figure.) This infinity allows ignoring boundaries of the weave as well as of the strands -- especially important in considering symmetries. One can think of the strands as adhering to the two sides of a plane except when passing through it between cells. That plane will be refered to as the {\em plane of the fabric} if the cells represent a fabric. In order to represent a fabric, the standard requirement \cite{GS1980} is that the weave represented {\em hang together}, not be separable into two or more layers, every strand of one layer passing over every strand of the other layer. If a weave hangs together, then it is a {\em fabric}; otherwise it is called a {\em pre-fabric}. Much of the literature on fabrics applies to pre-fabrics because hanging together is irrelevant to what is discussed. This paper will use fabrics as examples, but much of what is said would apply to pre-fabrics if one simply overlooks that they come apart. Their doing so would make the definition of the plane of a pre-fabric problematic.

\begin{figure}
\[
\begin{array}{cc}
\epsffile{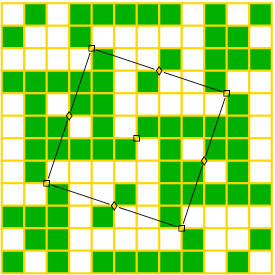} &\epsffile{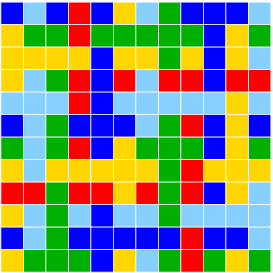}\\

\mbox{(a)} &\mbox{(b)}
\end{array}
\]
\caption{a. Order-20 example of species $33_4$ with lattice unit of $G_1$ outlined \cite[Figure 13b]{Thomas2010b}. b. Five-colouring by thin striping.}
\label{fig:76ab}
\end{figure}

For the sake of seeing them clearly, all fabrics are {\em coloured fabrics}, that is with strands that exhibit the same colour on both sides. The important convention of weaving diagrams uses the colour of strands, top-to-bottom strands, {\em warps}, dark and side-to-side strands, {\em wefts}, pale to signal in a diagram called a {\em design} which strand is over which strand from an arbitrary point of view well off the plane.%
\footnote{Matrices of numbers were used in the early paper \cite{HT1984}, which has errors as a result.} 
This is {\em normal colouring}, that of all the diagrams here in green and white. I refer to the side of the fabric seen thus as its {\em obverse} and the other side its {\em reverse}. When the reverse is shown, it is shown as it would look in a mirror set up on the back side of the fabric. This makes the cells of obverse and reverse correspond. While the reverse of a design is just its colour complement and so of no interest, the reverse of a weave coloured in any other way may differ seriously. Examples are Figures~\ref{fig:66abc}b and c, \ref{fig:68abc}b and c, \ref{fig:67abc}b and c, \ref{fig:82ab}, and \ref{fig:83ab}. There is more to do before we colour differently.

\begin{figure}
\[
\begin{array}{cc}
\epsffile{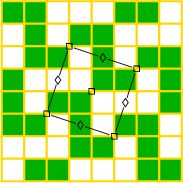} &\epsffile{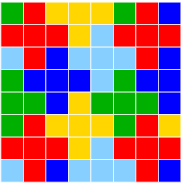}\\

\end{array}
\]
\caption{a. Roth's example of species $34$, 10-107-1, with lattice unit of $G_1$ marked \cite[Figure 8a]{Thomas2010b}. b. Five-colouring by thin striping.}
\label{fig:69ab}
\end{figure} 

\section{Symmetry}
\label{sect:sym}

\noindent
Much of the interest in weaving is symmetry \cite{GS1980}. Gr\"unbaum and Shephard define a symmetry of a fabric $F$ with plane $E$ as an isometry that maps each strand of $F$ into a strand of $F$ and either preserves all or reverses all of the rankings of the strands. A {\em symmetry} consists therefore of an isometry within $E$ (translation, rotation, reflection, or glide-reflection) possibly combined with three-dimensional reflection $\tau$ in the plane $E$, which reverses the ranking of the strands. The symmetries together form the {\em symmetry group} $G_1$ of the fabric under composition. The isometries alone form the {\em isometry group} $G_2$. Those symmetries that do not reverse the sides of the fabric form a subgoup $H_1$ of $G_1$ and $G_2$ of {\em side-preserving} symmetries. $H_1$ is either all of $G_1$ or is normal of index 2 because there is some {\em side-reversing} symmetry $\sigma$ such that $G_1 = H_1 \cup \sigma H_1$ \cite{GS1988}. The symmetries of $F$ determine the symmetries of its design. For the most part they are the same. One to watch is the quarter turn, which will have great importance in this paper. If one rotates the design of Figure~\ref{fig:66abc}a or \ref{fig:81b72b}a through a quarter turn at any of the centres indicated by the tiny black {\em boxes}, it is invariant, but it is not the design of the represented fabric rotated a quarter turn. This is because of the convention that warps are dark and wefts pale. The rotation has interchanged which is which, and so the design of the rotated fabric is the colour complement of the figure. We indicate with the little black box both the centre of the rotation and that $\tau$ has to be added to the rotation to restore the design. In this case a symmetry is a quarter turn composed with $\tau$. When a quarter turn does not need to be `corrected' in this way to be a symmetry of the design, it is represented by a hollow box where its centre is, as in Figures~\ref{fig:81a73b}a, \ref{fig:76ab}a, and \ref{fig:69ab}a. The other similar operations, which will be indicated only when they are symmetry operations, are half turns. Their centres are indicated by diamonds, black with they are combined with $\tau$ and hollow when $\tau$ is not needed. The latter are more common, the former not occurring until Figure~\ref{fig:75a}. Indicated centres of quarter turns, with or without $\tau$, are also centres of half turns without $\tau$, therefore not so indicated.

\begin{figure}
\[
\begin{array}{cc}
\epsffile{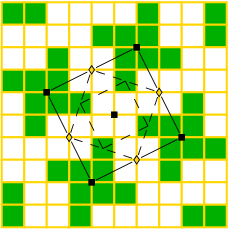} &\epsffile{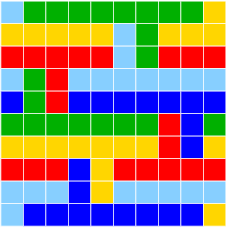}\\

\mbox{(a)} &\mbox{(b)}
\end{array}
\]
\caption{a. Roth's example of species $35_3$, 10-39-1, with lattice unit of $G_1=H_1$ outlined and lattice units of levels 1 and 2 dashed \cite[Figure 11a]{Thomas2010b}. b. Five-colouring by thin striping. Because quarter turns are all side-reversing, {\em on each side} only half-turn symmetry of the pattern is visible. Likewise Figures~\ref{fig:77ab}b, \ref{fig:68abc}b and c, \ref{fig:70ab}b, \ref{fig:71ab}b, and \ref{fig:89ab}a and b.}
\label{fig:81b72b}
\end{figure}

Translations as symmetry operations are not indicated on diagrams except incidentally because they can be seen. Mathematical weaving has been limited to designs periodic in two different directions. Because the sequence of pale and dark cells along each strand is periodic, I use the term {\em order} for those linear periods \cite{GS1980} except in the next sentence. Period parallelograms can be either square comprising where a linear period of consecutive strands meet a linear period of consecutive perpendicular strands, for example the whole of Figure~\ref{fig:81a73b}, 10 strands by 10 strands, or oblique, as outlined in the same figure with a thin black line. In either case those of interest have corners that form a lattice under the translation subgroup of the symmetry group; such parallelograms are called {\em lattice units}. The lattice units for the symmetries in Figures~\ref{fig:66abc} to \ref{fig:75a} and many beyond, which I outline with thin lines, are oblique squares (with centres of quarter turns at corners) on the hypotenuses of right-angled triangles with horizontal and vertical sides of length $M_i$ and $N_i$ \cite{Roth1993}, where $i = 1$, 2, 3, 4, or 5 is called the {\em level} of the lattice unit. A lattice unit of each level 2, 3, 4, or 5 is based on that of the next lower level by doubling of size. In \cite{Thomas2010b}, certain conventions for drawing these lattice units were established. Level-1 lattice units are bounded by lines in Figures~\ref{fig:66abc}a and \ref{fig:67abc}a, where it is shown that they can be drawn with either a cell-centre or cell-corner box in the centre representing the same lattice of quarter-turn centres ($M_1=2, N_1 =1$). Figure~\ref{fig:68abc}a shows another level-1 lattice unit ($M_1=3, N_1 =4$). This indicates that the frequent use of $M_1=2, N_1 =1$ is just to control size. Larger possibilities are discussed in Section~\ref{sect:larger}. It is the defining characteristic of level 1 that $M_1$ and $N_1$ are relatively prime with one odd and one even. Each higher-level lattice unit is drawn as a square escribing origami-style the square of the next level down, whether the inner square is drawn or not. Figures~\ref{fig:69ab}a, \ref{fig:70ab}a, and \ref{fig:71ab}a have level-2 lattice units illustrated without the level-1 square inscribed. Figures~\ref{fig:81a73b}a and \ref{fig:81b72b}a illustrate level-3 lattice units with the nested squares of levels 1 and 2 inscribed dashed. Figures~\ref{fig:81a73b}a, \ref{fig:69ab}a, \ref{fig:70ab}a, \ref{fig:85a74b}a, and \ref{fig:80} are based on the upper level-1 unit of Figure~\ref{fig:66abc}a. Figure~\ref{fig:71ab}a (level 2) is based on the lower level-1 unit of Figure~\ref{fig:66abc}a. Beyond level 2, only the former kind is possible \cite{Thomas2010b}. That is, lattice units of levels 3 to 5 have cell corners at their centres. Designs of level 4 are illustrated in Figures~\ref{fig:76ab}a, \ref{fig:77ab}a, \ref{fig:75a}, \ref{fig:90}, and \ref{fig:98}.
Level 4 is as far as $G_1$ lattice units for isonemal prefabrics can go \cite{Thomas2010b}.

\begin{figure}
\[
\begin{array}{cc}
\epsffile{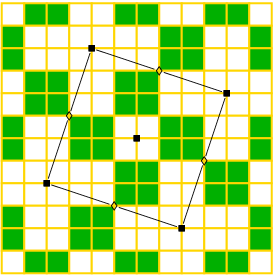} &\epsffile{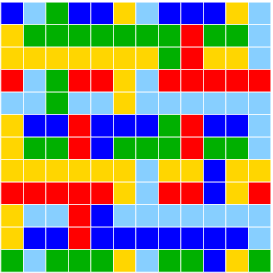}\\

\mbox{(a)} &\mbox{(b)}
\end{array}
\]
\caption{a. Fabric 10-85-1 of Figure~\ref{fig:71ab}a doubled, of species $35_4$, with lattice unit of $G_1$ outlined \cite[Figure 13a]{Thomas2010b}.  b. Five-colouring by thin striping.}
\label{fig:77ab}
\end{figure}

There is a reason to look also at level-5 lattice units. Important for some purposes are the $\tau$-free symmetry operations, which compose the side-preserving subgroup $H_1$ of the symmetry group $G_1$. Since the operations with black symbols in Figure~\ref{fig:90}, for example, all involve $\tau$, the side-preserving symmetry operations are the \whboxx s and $\blacksquare^2 = \diamondsuit$, and the lattice unit that they define is of level 5 (dashed like the units of lower levels) surrounding the lattice unit of the symmetry group of level 4. Likewise the side-preserving subgroup in Figure~\ref{fig:85a74b}a consists of \whboxx s and $\blacksquare^2 = \diamondsuit$ and is at level 4, while the symmetry group's lattice unit is at level 3. 
Figure~\ref{fig:80}, with $G_1$ at level 3, on the other hand has the level-4 lattice illustrated merely for the sake of illustrating a set of nested lattice units. Most of these illustrations are based on the level-1 lattice unit with $M_1=2$ and $N_1=1$ for the sake of making them as small as possible. $M_2= M_1 + N_1$ and $N_2= M_1 - N_1$, both odd, are always also relatively prime \cite{Thomas2010b}. From level three on the pattern is simple and the same: $M_3 = 2M_1$ and $N_3 = 2N_1$, $M_4 = 2M_2$ and $N_4 = 2N_2$, such lattice units being four times the area of that two levels down. The important designs should now make sense and be distinguishable. Besides the choice of which cells are pale and dark, Figures~\ref{fig:69ab}a (species 34) and \ref{fig:70ab}a (species 36) differ in which operations occur at the corners and centres of the lattice units. Figures~\ref{fig:70ab}a ($36_2$) and \ref{fig:71ab}a ($36_s$), on the other hand, differ in where the same operations occur, the latter taking advantage of the possibility noted above that level-2 lattice units can be centred on cell centres as well as cell corners. 

\begin{figure}
\[
\begin{array}{ccc}
\epsffile{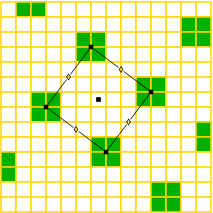} &\epsffile{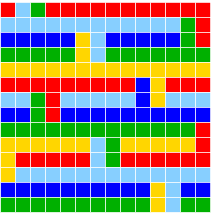} &\epsffile{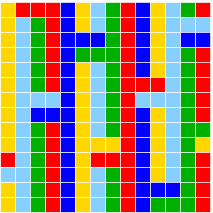}\\

\mbox{(a)} &\mbox{(b)} &\mbox{(c)}
\end{array}
\]
\caption{a. Order-25 fabric of species $36_1$ with lattice unit of $G_1$ marked. b. Five-colouring by thin striping. Obverse. c. Reverse. Cf.~Figure~\ref{fig:80}.}
\label{fig:68abc}
\end{figure}

The other symmetry operations, reflections and glide-reflections, are indicated in my diagrams by locating their axes with thick solid and broken lines respectively. Figures~\ref{fig:Perp}a and \ref{fig:5cXX}b have both. Mirrors (axes of reflection) always have the effect of interchanging warps and wefts and so of reversing the colours of the design. As this is not wanted for mirror symmetry, $\tau$ is always needed to restore the colours of the design. Glide-reflections are not so simple. Those indicated in Figure~\ref{fig:6-1-1} take diagonally adjacent pairs of dark cells to nearby similar pairs, but because glide-reflection also reverses warps and wefts, $\tau$ is needed to make the image pair dark rather than pale. These are {\em side-preserving glide-reflections}. On the other hand, the glide-reflections in Figure~\ref{fig:5cXX}a take pale areas to dark areas and vice versa. No $\tau$ is needed, and so these {\em side-reversing glide-reflections} are indicated by hollow broken lines. Figure~\ref{fig:5cXX}b shows that the two sorts of glide-reflection axes can occur parallel in the same design and also indicates the final notational convention for axes. Perpendicular to the axes already explained are (in $H_1$) axes of both side-preserving glide-reflection interchanging pale and dark zig-zags and (in $G_1$ but not $H_1$) of mirror symmetry within each zig-zag and the design as a whole. This double nature is indicated by breaking the filling of the mirrors. Thin lines in this figure and elsewhere just indicate boundaries of an area of interest, in this case and Figure~\ref{fig:6-1-1} lattice units.

There is a subtlety to glide-reflections lacking in mirrors. In order to be a symmetry operation, a diagonal mirror has to pass through every cell from corner to corner. Not so diagonal axes of glide-reflection. While such axes in all diagrams in this paper do run through cells from corner to corner, there are several subspecies characterized by the axes running through cells from mid-side to adjacent mid-side. This is possible because the cells through which an axis passes are translated along the axis by the glide to other cells where they can fit reflected. I refer to this location as {\em not in mirror position}. This possibility will matter later but only by being excluded.

\begin{figure}
\[
\begin{array}{cc}
\epsffile{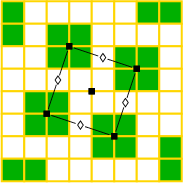} &\epsffile{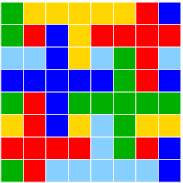}\\

\mbox{(a)} &\mbox{(b)}
\end{array}
\]
\caption{a. Roth's example of species $36_2$, 10-27-1, with lattice unit of $G_1$ marked \cite[Figure 8b]{Thomas2010b}. b. Five-colouring by thin striping.}
\label{fig:70ab}
\end{figure}

With symmetry operations and their groups defined, it makes sense to introduce the seminal innovation of Gr\"unbaum and Shephard \cite{GS1980}. They realized that weaving would be both more attractive and mathematically interesting if attention were focused on {\em isonemal} fabrics, which are those on which the symmetry group is transitive on the strands. A strand can be transformed to any other strand by a symmetry operation. Such fabrics are what they listed in their two catalogues \cite{GS1985,GS1986}. Richard Roth gave complete algebraic analysis in \cite{Roth1993} of the symmetry groups of isonemal fabrics of order greater than 4. (Fabrics of the smallest orders are anomalous in a variety of ways.) The 39 types into which Roth divided the symmetry groups necessarily overlap (some are subgroups of others), and I have found it useful instead to divide the fabrics into non-overlapping {\em species} according to which group type applies, several of them subdivided. The first ten species have only parallel axes of symmetry \cite{Thomas2009}: 1, side-preserving glide-reflections; 2, side-reversing glide-reflections; 3 and 4, both sorts of glide-reflection axes differently positioned; 5, mirrors; 6 and 7, mirrors in $G_1$ that turn into axes of side-preserving glide-reflection axes in $H_1$, differently spaced; 8, reflections and side-reversing glide-reflections; 9 and 10, reflections and side-preserving glide-reflections, differently spaced. Species 11 to 32 \cite{Thomas2010a} have symmetry groups with perpendicular axes of symmetry. Fabrics in some of them will be mentioned in Section~\ref{sect:nonsatin}, but it is not necessary to discuss them all. The species of chief importance in this paper are 33 to 39, all with quarter-turn symmetry and no axes of symmetry (called by crystallographers and Roth $p4$), as the main results are about fabrics in those species.

\begin{figure}
\[
\begin{array}{cc}
\epsffile{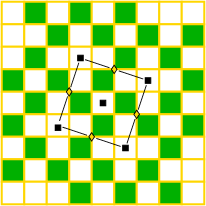} &\epsffile{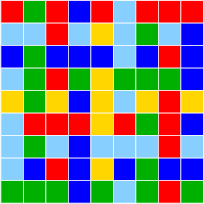}\\

\mbox{(a)} &\mbox{(b)}
\end{array}
\]
\caption{a. Example of species $36_s$, 10-85-1, with lattice unit of $G_1$ marked \cite[Figure 9]{Thomas2010b}. b. Five-colouring by thin striping, the same pattern as Figure~\ref{fig:67abc}c.}
\label{fig:71ab}
\end{figure}

\section{Isonemal fabrics, mainly species 33 to 39}
\label{sect:isonemal}

\noindent
Let us look at the species of particular interest \cite{Thomas2010b}. Species 33 is divided into 2 subspecies (indicated like all subspecies by subscripts), $33_3$ and $33_4$, depending on whether their lattice units, having {\em side-preserving quarter-turn centres \whbox at corners and centre}, are at level 3 or 4. See Figures~\ref{fig:81a73b}a and \ref{fig:76ab}a. The levels and italicized feature are definitive of species 33. Species 35 is divided like 33 into 2 subspecies, $35_3$ and $35_4$ according to the level of lattice unit, having {\em side-reversing quarter-turn centres \blbox at corners and centre}. See Figures~\ref{fig:81b72b}a and \ref{fig:77ab}a. Again the levels and italicized feature are definitive of species 35. Because of their levels the quarter-turn centres in these species must occur at corners of cells. That is not the case at levels 1 and 2; centres can be at cell centres too. And so the subspecies of species 36 are more delicately differentiated. Species 36 is divided into 3 subspecies, having almost the same feature description as 35 but at lower levels: at level 2 subspecies $36_2$ (see Figure~\ref{fig:70ab}a) has the centres at cell corners like 33 and 35, but subspecies $36_s$ (see Figure~\ref{fig:71ab}a) has its quarter-turn centres at the centres of cells. The point made in Figure~\ref{fig:66abc}a is that the corners and centre of level-1 lattice units can be differently placed. If the lattice-unit corners are at cell corners, then its centre is at a cell centre and vice versa. Subspecies $36_1$ has lattice-unit corners and centres in those positions -- the same kind of lattice units differently placed (see Figures~\ref{fig:68abc}a and \ref{fig:66abc}a). Some of the information of this paragraph -- and more to come -- is summarized in Table~1.
There is one more species with the same kind of quarter-turn centre at corners and centre of lattice units. Species 34 has, like 33, {\em side-preserving quarter-turn centres \whbox at corners and centre} but at level 2. See Figure~\ref{fig:69ab}a, where the level-2 lattice unit, like those in Figure~\ref{fig:70ab}a and \ref{fig:71ab}a, is a dashed square in Figure~\ref{fig:81a73b}a. One naturally wonders whether the side-preserving quarter-turn centres like those of species 34 can be placed all in the centres of cells; this is the sort of thing Roth investigated. In this case there is an easy explanation: side-preserving quarter-turn centres can't occur in the centre of a cell on account of their effect on the cells of their location. (The cell must remain the same colour without the use of $\tau$ after a quarter turn, which is self-contradictory.)

\begin{table}
\begin{tabbing}
Species \hskip 5 pt\= 39\hskip 5 pt \= $36_1$\hskip 5 pt\= $36_2$\hskip 5 pt\= $36_s$\hskip 5 pt\= 34 \hskip 5 pt \= 38\hskip 5 pt\=$33_3$\hskip 5 pt\=$35_3$\hskip 5 pt\=$33_4$\hskip 5 pt\=$35_4$\hskip 5 pt\=37\\
\phantom{m}\\
Level \> 1  \> 1 \> 2 \> 2 \> 2  \> 3 \>3 \>3 \>4 \>4 \>4\\
Unit\\
Centre \> \whbox  \> \blbox \> \blbox \> \blbox \> \whbox  \> \whbox \>\whbox \>\blbox \>\whbox \>\blbox \>\whbox\\
Unit\\
Corners \> \blbox  \> \blbox \> \blbox \> \blbox \> \whbox  \> \blbox \>\whbox \>\blbox \>\whbox \>\blbox \>\blbox\\
Mid-\\
Side \> \diab  \> \dia \> \dia \> \dia \> \dia  \> \diab \>\dia \>\dia \>\dia \>\dia \>\diab\\
\end{tabbing}
\noindent Table 1. Correspondence among species, levels, and symmetry operations (as illustrated).
\label{tab:boxes}
\end{table}

The two identical columns in Table~1 
reveal that two subspecies, $36_2$ and $36_s$, have the same symmetry group $G_1$, abstractly -- even geometrically -- considered, whose difference lies in their distinct placement in the cells of the design (Figures~\ref{fig:70ab}a and \ref{fig:71ab}a).

\begin{figure}
\[
\epsffile{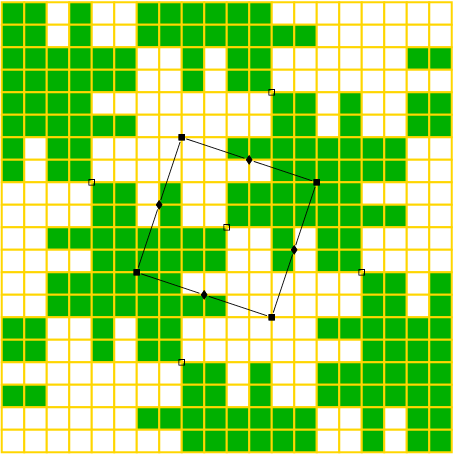}
\]
\caption{Roth's species-37 example 20-3391 \cite[Figure 7]{Roth1993}, with one lattice unit of $G_1$ outlined and the corners of one larger lattice unit of $H_1$ marked.}
\label{fig:75a}
\end{figure}

We now move on to side-preserving and side-reversing quarter turns mixed. At level 1, where the lattice unit can be of either sort as with species $36_1$ in Figure~\ref{fig:66abc}a, the species is 39 (see Figure~\ref{fig:67abc}a, where both sorts of $G_1$ lattice unit are illustrated again). At level 3 the species is 38 (see Figure~\ref{fig:85a74b}a), and at level 4 it is 37 (see Figure~\ref{fig:75a}). In species 37 and 38 there are also alternatives; Figure~\ref{fig:75a} could just as well be drawn with centres of side-reversing quarter turns \blbox at the centres of the lattice units and the centres of side-preserving quarter turns \whbox at their corners. Likewise Figure~\ref{fig:85a74b}a of species 38. And this arbitrariness matters. As we shall see, two ways of doing things result (Thm 2).

\begin{figure}
\[
\epsffile{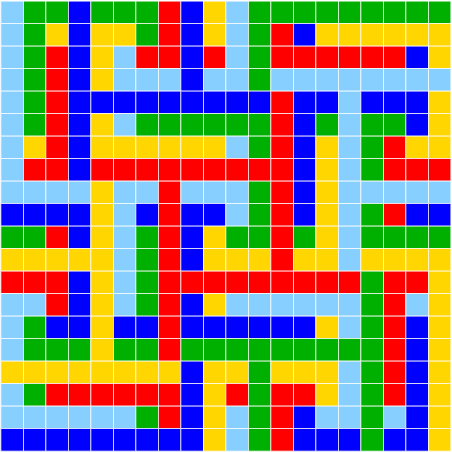}
\]
\[
\epsffile{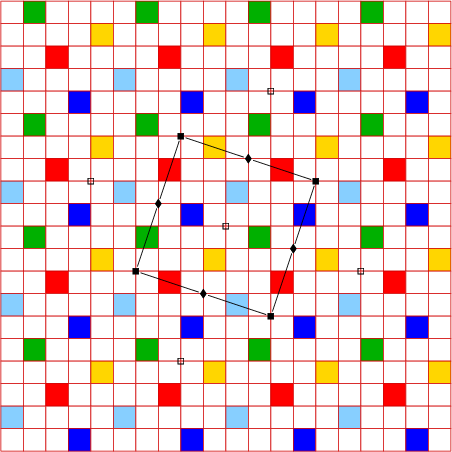}
\]
\caption{Five-colouring by thin striping of fabric of Figure~\ref{fig:75a}, and below it its redundant cells with symmetries from Figure~\ref{fig:75a}.}
\label{fig:75b}
\end{figure}

Having looked at the species of main interest here, we can look at some families of examples beyond the figures so far referred to. The simplest isonemal fabrics -- introduced by Gr\"unbaum and Shephard \cite{GS1980} -- are twills, a {\em twill} being defined as having a design in which each row looks like the row below it shifted by one cell (all of these shifts being in the same direction). As a consequence the columns look the same but shifted too. The simplest twills have one dark cell per order length of the strand, and there is only one such twill per order. No twill is needed in this paper. The second-simplest configuration is what is called the {\em doubled twill}, each strand of the twill being replaced by two strands behaving the same and therefore each cell of the twill being replaced by a {\em block} of four cells of the same colour. This is no longer a twill itself, but there is still one dark block per order. No doubled twill is needed here, but the process of doubling and its result is mentioned in a number of places. See Figures~\ref{fig:77ab}, \ref{fig:80}, and \ref{fig:89ab}a.

\begin{figure}
\[
\begin{array}{cc}
\epsffile{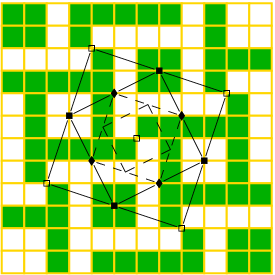} &\epsffile{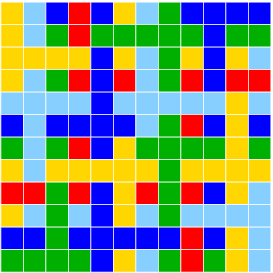}\\

\mbox{(a)} &\mbox{(b)}
\end{array}
\]
\caption{a. Roth's species-38 example 20-19437 \cite[Figure 8]{Roth1993}, with one lattice unit of $G_1$ and the larger lattice unit of $H_1$ outlined and lower-level lattice units dashed \cite[Figure 12a]{Thomas2010b}. b. Five-colouring by thin striping.}
\label{fig:85a74b}
\end{figure}

The next-simplest fabrics are the satins, the other kind of fabric discussed in Gr\"unbaum and Shephard's initial paper \cite{GS1980}. Another name for a {\em satin} is `rearranged twill' because to make the next row of a satin one displaces each row, having one dark cell per order length, by more than the one that would make a twill, always in the same direction. The simplest, denoted (5,3), has order 5 and displacement equivalently 2 or 3. Satins are of three kinds depending on how the two-dimensional arrangement of the dark cells strikes one's eye. If they are the corners of squares as in Figure~\ref{fig:66abc}a, then the satin is called {\em square}, and if they make rhombs or rectangles they are called {\em rhombic} or {\em rectangular}. Rhombic satins are of species 28 with no quarter-turn symmetry \cite[Thm 2]{Thomas2010a}. Rectangular satins are all of species 26 with no quarter-turn symmetry \cite[Thm 1]{Thomas2010a}. Square satins of odd and even orders fall into species $36_1$ and $36_s$ respectively \cite{Thomas2010b}. It is unfortunate that there are no satins of order 3, 4 or 6 \cite{GS1985}, but order 5 is small enough to give attractive results. The satin that is useful for this purpose, (5,3), is square and of odd order. The restriction to square satins, and so to species $36$ is obvious, a rhomb or rectangle with the correct symmetry is actually a square. The next odd square-satin order after 5 is 13 (17, 25, 29, 37,\dots). Oddness of order is not necessary, but the smallest square satin of even order is (10,3) of species $36_s$.

It is perhaps useful to note the many different squares here. In order of size, there are the hollow and filled boxes locating centers of quarter turns, cells, oblique lattice units of species 33 to 39, and order-by-order lattice units like Figure~\ref{fig:81a73b}a. There are also the eye-catching squares that give square satins their name, in Section~\ref{sect:col} lattices of redundant cells of the same colour, in Section~\ref{sect:nonplanar} period parallelograms that can be made into flat tori, and finally in Figures~\ref{fig:colCubeFront} and \ref{fig:colCubeBack} the faces of cubes, but these are all lattice units.

\section{Colouring}
\label{sect:col}
\begin{figure}
\[
\begin{array}{ccc}
\epsffile{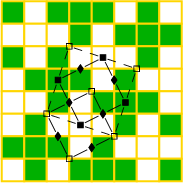} &\epsffile{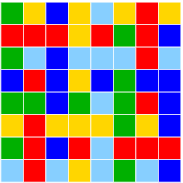} &\epsffile{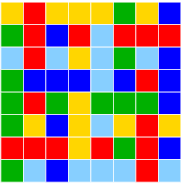}\\

\mbox{(a)} &\mbox{(b)} &\mbox{(c)}
\end{array}
\]
\caption{a. Roth's example of species 39, 10-93-1 with the larger lattice unit of $H_1$ and two different lattice units for $G_1$ marked \cite[Figure 7]{Thomas2010b}. b. Five-colouring by thin striping. Obverse. c. Reverse. The same pattern as Figure~\ref{fig:71ab}b.}
\label{fig:67abc}
\end{figure}

\noindent
Except for such information as is needed about species 11 to 32 in Section~\ref{sect:nonsatin}, which will be provided there, that is all that is needed from the classification of isonemal fabrics. We can now turn to colouring with more than the coding colours, which will still be used for design-identification purposes. All weaving diagrams are coloured; otherwise they would be blank. They are coloured normally to indicate their topology, but the strands can be coloured any way at all for the sake of appearance. A standard definition in the business of colour is {\em perfect colouring} of an object, which is a choice of colours for subunits -- in this case strands -- that makes all symmetries of the object permute the colours of the subunits coherently. If one red strand is transferred to a yellow strand, then all red strands are transferred to yellow strands; these are called {\em colour symmetries}. The simplest and most attractive colouring is perfect colouring in which all symmetries are colour symmetries; in his second weaving paper \cite{Roth1995}, Roth studied chromatic 2-colouring as well, which is less constrained. His papers were a response to my initial investigation with J.A.~Hoskins of perfect colouring with two colours \cite{HT1991}. Perfect colouring is easy if all strands are different colours, but one wants to use a finite number of colours. We learned that this was only possible if strands were arranged in stripes periodically, either {\em thinly} (colour${}_1$, colour${}_2$, \dots , colour${}_p$, colour${}_1$, \dots) or {\em thickly} (colour${}_1$, colour${}_1$, colour${}_2$, colour${}_2$, \dots , colour${}_p$, colour${}_p$, colour${}_1$, \dots) \cite[Thm 1]{HT1991}. The warps and wefts could use different sets of $p$ colours; it is an aesthetic decision to make them the same here to minimize the number of colours and produce monochrome motifs. Every diagram not coloured green and white in this paper has its strands coloured with a colour that can be easily determined by looking at the diagram. Whenever there are three consecutive cells of one colour, then that warp or weft has that colour. Two adjacent cells of a single colour do not indicate that if striping is thick. Figure~\ref{fig:82ab}b has blocks of four cells in light and dark blue, red, and green that indicate nothing, but the pinwheel shapes indicate intersecting pairs of thick stripes of the same colour. It is the four cells in a row that indicate the colour of each strand of the thick stripes. The stripe colours are more obvious in Figures~\ref{fig:82ab}a, \ref{fig:83ab}, and \ref{fig:87ab}a, less obvious in \ref{fig:87ab}b. There is something important to pay attention to about the blocks where stripes cross whether thick or thin.

\begin{figure}
\[
\epsffile{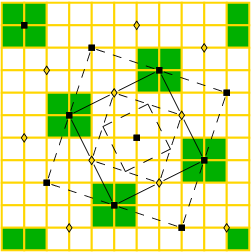}
\]
\caption{Fabric 10-3-1, the $(5,3)$ satin (Figure~\ref{fig:66abc}a) doubled, its dark blocks being one set of redundant blocks for thick colouring by 5 colours. Its level-3 $G_1$ lattice unit is outlined, and the corresponding lattice units of levels 1, 2, and 4 are dashed.}
\label{fig:80}
\end{figure}

Where stripes of the same colour cross it does not matter which of the two is the `visible' one; therefore how such cells are woven has no influence on the appearance of the woven material.
Such cells are called {\em redundant} and the other cells {\em irredundant}. A lack of redundancy is achieved by using different sets of colours for wefts and warps. Normal colouring is a special case of this, but it can be done with any larger even number of colours divided equally between warps and wefts. Interest in such patterns is limited by the fact that the only way that cells of the same colour can form a motif by adjacency is by being in a line \cite[Figures 7, 8]{Thomas2014}. When the striping is thick, motifs can form from the rows of adjacent same-colour pairs, which allows some interesting variations \cite[Figures 4, 5]{Thomas2014}. Visual interest is actually increased by having strands in different directions be of the same colour, and redundancy is a consequence. With only two colours, half of the cells are redundant, so that all of the patterns for a given colouring are half the same, but the other half gives room for much variation. As the number of colours $p$ increases, the proportion of cells $1/p$ that are redundant, whether the stripes are thin or thick, diminishes, and so there is more freedom for variation. Unfortunately this aesthetic benefit is countered by the disadvantage of more and more colours. I find six colours the absolute maximum that is attractive, something that will differ from person to person. 

\begin{figure}
\[
\begin{array}{cc}
\epsffile{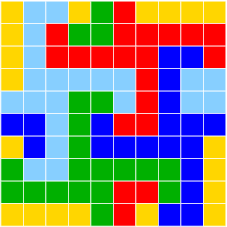} &\epsffile{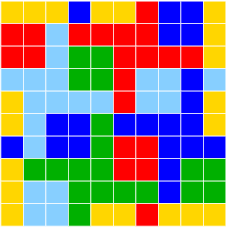}\\

\mbox{(a)} &\mbox{(b)}
\end{array}
\]
\caption{a. One 5-colouring by thick striping of 10-55-2 (Species $33_3$, Figure~\ref{fig:81a73b}a). a. Obverse. b. Reverse.}
\label{fig:82ab}
\end{figure}

No matter how many redundant cells of a colouring there are or how they are distributed, they will form a configuration. Because to be colour symmetries the symmetries of the fabric coloured must each carry redundant cells to redundant cells and collectively they must be transitive on the strands \cite[Lemma 4.6]{Thomas2013}, the redundant cells of a perfect colouring of an isonemal fabric must themselves be the dark cells of an isonemal design. Call it the {\em derived prefabric}. In order to state this circumstance with precision and generality it is necessary to use the isometry group $G_2$ of a prefabric. As explained above, elements of the group $G_1$ consist of isometries in the plane and maybe $\tau$. They are therefore of the form $(i,r)$ where $i$ is an plane isometry and $r$ is either the identity reflection $e$ in the plane or $\tau$, and composition is of the direct-product type. These elements can all be projected onto $G_2$, the set of elements of the form $(i,e)$, which are not necessarily symmetries of the fabric but could be its side-preserving subgroup $H_1$ if there were no symmetries of the form $(i,\tau)$. A necessary condition on the derived prefabric of a colouring of a fabric for the colouring to be perfect is that $G_2$ of the fabric coloured must be a subgroup of $G_2$ of the derived prefabric \cite{Thomas2013}. Any isometry that is the isometric part of a symmetry operation must be the isometric part of a symmetry of the derived prefabric, as only location of redundant cells matters.

\begin{figure}
\[
\begin{array}{cc}
\epsffile{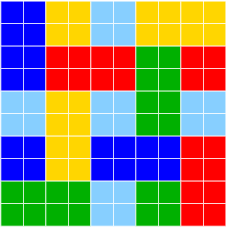} &\epsffile{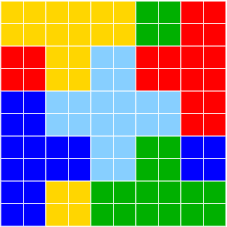}\\

\mbox{(a)} &\mbox{(b)}
\end{array}
\]
\caption{Second 5-colouring of 10-55-2 by thick striping (species $33_3$). a. Obverse. b. Reverse.}
\label{fig:83ab}
\end{figure}

The simplest redundancy configuration is the checkerboard of redundant cells when there are two colours. The simplest redundancy configurations for thin striping are the simplest twills of order equal to the number of colours with one dark cell per order, the colours being the same in both directions, and there is only one such twill per order. This was the subject of \cite{Thomas2014}. For 3 and 4 colours these twills are all there are. The second-simplest configuration, the {\em doubled twill}, is useful for thick striping. This was the subject of \cite{Thomas2013}. The two papers referred to discuss colouring with twilly redundancy. With both of these redundancy configurations, the order of the colours vertically and horizontally is the same (or opposite depending on how one counts). Beginning with 5 colours, there begins to be possible non-twilly redundancy as well. That is what this paper is about. The configurations of redundant cells that allow this are the $(5,3)$ satin of Figure~\ref{fig:66abc}a for 5 colours and the fabric 6-1-1 of Figure~\ref{fig:6-1-1} for 6 colours, each the smallest of infinite families of no practical interest because of involving more and more colours. An infinite family that the $(5,3)$ satin begins is the square satins of odd orders in species $36_1$. A contrast with twilly redundancy is that now the orders of the colours vertically and horizontally are not the same; this is illustrated in most of the coloured figures.

\begin{figure}%
\[
\begin{array}{cc}
\epsffile{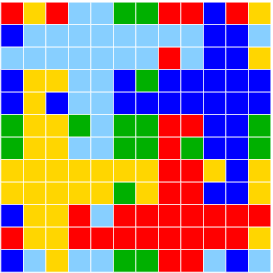} &\epsffile{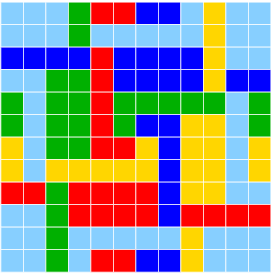}\\
\mbox{(a)} &\mbox{(b)}
\end{array}
\]
\caption{Five-colourings by thick striping of order-20 fabric (species $33_4$, Figure~\ref{fig:76ab}a).}
\label{fig:87ab}
\end{figure}                    

Recall that, for a choice of colours for strands to be a perfect colouring, all symmetries of the design permute the colours of the strands coherently; that is, if one red strand is taken by a design symmetry to a yellow strand, then all red strands are taken to yellow strands. This is more easily thought about in terms of redundant cells. Because both strands passing through a redundant cell are of its colour, the mapping of strands' colours is determined by the mapping of redundant cells. If one red redundant cell is taken to a yellow redundant cell, than all must be, as must be the strands through them. This means that to check that a colouring is perfect one need look only at redundant cells, which relieves one from having to check the other side of the coloured fabric since the redundant cells appear with the same colour on both sides. 

\begin{figure}
\[
\begin{array}{cc}
\epsffile{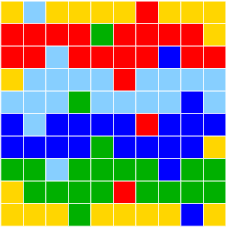} &\epsffile{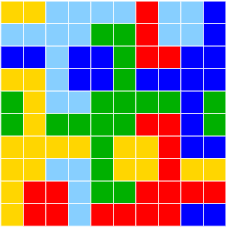}\\

\mbox{(a)} &\mbox{(b)}
\end{array}
\]
\caption{Distinct 5-colourings of 10-39-1 (species $35_3$, Figure~\ref{fig:81b72b}a) by thick striping.}
\label{fig:84ab}
\end{figure}

It is good to see what is actually involved in this specification. Figure~\ref{fig:75b} indicates both a full colouring above and the redundant cells coloured with other cells white below. The periodicity of the colouring, making the redundant cells of each colour a square lattice, means that if a red cell is placed on a yellow cell by an isometry, all the red cells are placed on yellow cells. The symmetries of the derived design (the $(5,3)$ satin) that take red cells to yellow cells take the four cells closest to the red ones, light blue, dark blue, yellow, and green, to those closest to yellow, light blue, green, dark blue, and red, in one of four positions depending on rotation. This is true for all of the symmetries $G_2$ of the derived fabric, and the necessary condition that $G_2$ of the fabric to be coloured is a subgroup of $G_2$ is sufficient to make the colour mapping of the redundant cells coherent. What does this show?
The necessary condition on the colouring of strands for perfect colouring, that $G_2$ of the design to be coloured be a subgroup of $G_2$ of the derived design, is also sufficient when the derived design is the square $(5,3)$ satin, the number of the colours equals the order of the derived design, and the colours are assigned to strands periodically. We stick to this situation in Sections~\ref{sect:col} to \ref{sect:larger}.

Note that, while the centre of rotation (modulo translation) in the discussion of Figure~\ref{fig:75b} pertains to rotations of the colouring whose centre lies in the centre of a cell, the same coherence obtains if the centre of rotation (modulo translation) lies not there but in the centre of the square with vertices at four neighbouring cell centres. Staying near the bottom left corner of the redundant-cell diagram, consider rotation about the centre \whbox of the small square with vertices in cells coloured red, yellow, dark blue, and green. That centre is at cell corners and permutes the 4 colours mentioned cyclically. But what happens to the fifth colour? Note that the 4 closest light-blue cells form a bigger square with the same centre; accordingly they are permuted, as are all light-blue cells, with no effect on their colour.

\begin{figure}
\[
\begin{array}{cc}
\epsffile{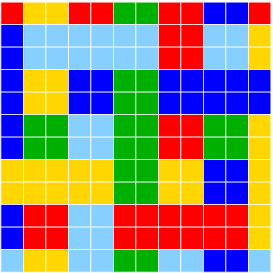} &\epsffile{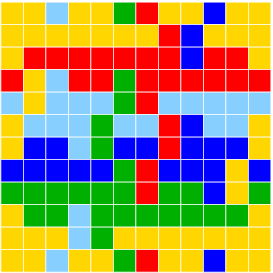}\\
\mbox{(a)} &\mbox{(b)}
\end{array}
\]
\caption{Five-colourings by thick striping of the fabric 10-85-1 doubled (species $35_4$, Figure~\ref{fig:77ab}a). a. That of Figure~\ref{fig:71ab}b doubled. b. Not that of Figure~\ref{fig:71ab}b doubled.}
\label{fig:89ab}
\end{figure}

The above discussion in terms of red and yellow applies to any pair of colours in the situation discussed based on that square satin. Whether the discussion depends on the oddness of that order lies outside the scope of this paper. When the order is even, the symmetry groups involved are of the same type, 36, and are arranged the same way with the centre of the lattice unit at a centre of $p4$ symmetry in the centre of a cell. This is true of both subspecies $36_1$ (Figure~\ref{fig:68abc}a) for odd-order satins and $36_s$ (Figure~\ref{fig:71ab}a) for even-order satins. (Groups of this type do not need to be so located. Subspecies $36_2$, illustrated in Figure~\ref{fig:70ab}a, has all of its quarter-turn symmetries at cell corners.)

I refer to the result of colouring the strands of a fabric as a {\em pattern} so that we can refer to the resultant pattern of a colouring of a fabric. Designs are a special case. It is patterns that are shown then in all of the weaving diagrams of figures from \ref{fig:66abc} to \ref{fig:13aXX}. A feature that is of aesthetic importance in patterns but not formally defined is motifs. They are what one's eye sees repeated when one looks at the pattern. In Figure~\ref{fig:66abc}b and c one sees stripes. In Figure~\ref{fig:81a73b}b one can see blocks of 4 differently multicoloured and bicoloured rectangles of 8 cells. Each of the former can be seen to be surrounded by four of the latter, making up another motif, but these larger motifs overlap, which seems not to be an aesthetic advantage. Sometimes the noticeable motif, like the 7-cell crosses of Figure~\ref{fig:76ab} or the 9-cell crosses of Figure~\ref{fig:70ab}, do not cover the pattern but occupy only part of it. And sometimes the motif, like the 5-cell crosses of Figure~\ref{fig:69ab}b or 5-block crosses of Figure~\ref{fig:83ab}b, cover the whole pattern. The fewer and simpler the motifs and the more of the pattern they cover, the more attractive the pattern seems to me to be. Thick striping allows more scope for interesting motifs; see Figures~\ref{fig:82ab}a and b, \ref{fig:91}, \ref{fig:92}, and \ref{fig:94}.
Whether one regards as a motif the squares, the presence of whose corners give square satins their name, is a matter of taste.

\section{Main results}
\label{sect:results}

\noindent
This paper introduces the topic of patterns of periodically coloured isonemal fabrics with redundancy configurations different from those of twills and doubled twills discussed in \cite{Thomas2013,Thomas2014}. Some of what can be done can now be stated.

\begin{thm}
In each subspecies of species $33$ to $39$ of isonemal fabrics there are orders such that all fabrics of those orders can be perfectly coloured by thin striping with 5 colours and redundant cells arranged as the satin of order 5, the area of the level-1 lattice unit on which the fabric is based.
\label{thm:sat1}
\end{thm}

The existence of such orders is illustrated by the patterns of thinly striped fabrics in Figures~\ref{fig:81a73b}b and \ref{fig:76ab}b for $33_3$ and $33_4$, Figure~\ref{fig:69ab}b for 34, Figures~\ref{fig:81b72b}b and \ref{fig:77ab}b for $35_3$ and $35_4$, Figures~\ref{fig:68abc}b and c for $36_1$, Figure~\ref{fig:70ab}b for $36_2$, Figure~\ref{fig:71ab}b for $36_s$,
Figure~\ref{fig:75b} for 37, Figure~\ref{fig:85a74b}b for 38, and Figure~\ref{fig:67abc}b and c for 39. The method used here for 5 colours appears to be general and to apply at least to larger odd orders, $M_1^2 + N_1^2 = 13$, 17, \dots , and perhaps to even orders, but none of that will be explored in this too long paper.

\begin{figure}
\[
\epsffile{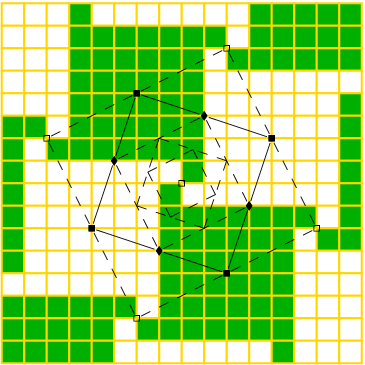}
\]
\caption{Order-20 example of species 37 with lattice unit of $G_1$ outlined and larger $H_1$ and lower-level lattice units dashed.}
\label{fig:90}
\end{figure}

Thick striping is also possible. The redundancy configuration analogous to the $(5,3)$ satin is that satin doubled (Figure~\ref{fig:80}). All square satins can be doubled \cite{Thomas2010b}. Table 2 
indicates what level of lattice unit goes with which species.
\begin{table}
\begin{tabbing}
\hskip 80 pt \= Level 1\hskip 15 pt \= Level 2\hskip 18 pt\= Level 3\hskip 15 pt\= Level 4\\
Symmetries\\
All \whbox        	\>$\emptyset$ \>34           \>$33_3$   \>$33_4$\\
All \blbox        	\>$36_1$      \>$36_2, 36_s$ \>$35_3$   \>$35_4$\\
Both \whbox and \blbox  \>39          \>$\emptyset$  \>38       \>37\\
\end{tabbing}
\vskip - 18 pt
\noindent Table 2. Correspondence among species, symmetry operations, and levels.
\label{tab:corr}
\end{table}
The group $G_2$ of the doubled (5,3) satin is at level 3, a proper subgroup of groups at levels 1 and 2, and so the group $G_2$ of no fabric at level 1 or 2 can be a subgroup of the group $G_2$ at level 3 based on the same level-1 lattice unit. This means that the lattice units of species 34, 36, and 39, which the table from \cite{Thomas2010b} shows are at those levels, are not of the right size to fit the redundancy configuration of the doubled (5,3) satin. Accordingly they cannot be used. This will justify their exclusion from the theorem indicating which species allow perfect colouring by thick striping.

\begin{thm}
For each subspecies of species $33$, $35$, $37$, and $38$ of isonemal fabrics based on a level-1 lattice unit of area 5 there are orders such that all fabrics of those orders can be perfectly coloured by thick striping in two different ways with 5 colours and redundant blocks arranged as the satin of order 5 doubled.
\label{thm:2sat1}
\end{thm}

\begin{figure}
\[
\epsffile{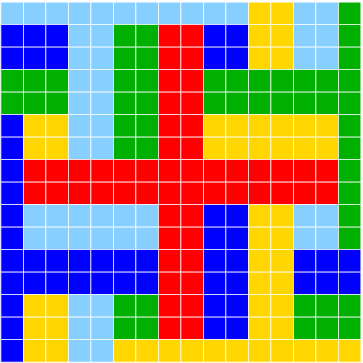}
\]
\caption{Five-colouring of fabric of Figure~\ref{fig:90} (species 37) by thick striping.}
\label{fig:91}
\end{figure}
\begin{proof}
The theorem is shown to be true partly by examples of patterns, Figures~\ref{fig:82ab}a and b and Figures~\ref{fig:83ab}a and b for $33_3$, Figures~\ref{fig:87ab}a and b for $33_4$, Figures~\ref{fig:84ab}a and b for $35_3$, Figures~\ref{fig:89ab}a and b for $35_4$, Figures~\ref{fig:91} and \ref{fig:92} for 37, and Figures~\ref{fig:86ab}a and b for $38$. The other part is to show why it is always possible to stripe fabrics of these species thickly in two ways.

The reason for this is slightly different for levels 3 and 4. In the case of colouring at level 3 (Figures~\ref{fig:82ab} and \ref{fig:83ab} for $33_3$, Figure~\ref{fig:84ab} for $35_3$, and Figure~\ref{fig:86ab} for $38$), one needs to refer to the designs, Figures~\ref{fig:81a73b}a, \ref{fig:81b72b}a, and \ref{fig:85a74b}a respectively. In each of those designs, a level-3 $G_1$ lattice unit of the same size and orientation is delineated by solid lines. (In the last case the $H_1$ lattice unit, twice the size, is also marked.) In each case -- and in the case of any level-3 lattice unit based on the same level-1 unit -- the four-cell blocks centred on the corners of the lattice unit are spaced like a doubled $(5,3)$ satin. One way to colour fabrics of those species {\em based on the same level-1 unit} is to make those corner blocks redundant. They will then be moved around by symmetries in the required way. But at level 3 there are just as good alternative lattice units taking their corners to be the centres of the lattice units used before. These are not marked, but they are always there. So a second set of redundant blocks are at the centres of the lattice units already discussed. They are always different in the sense intended -- the sense in which quadratic equations always have 2 roots. It is logically possible, so far as I know, for the two colourings to look the same like the unrelated Figures~\ref{fig:71ab}b and \ref{fig:67abc}c.

\begin{figure}
\[
\epsffile{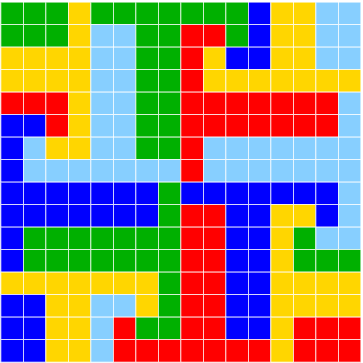}
\]
\caption{Five-colouring of fabric of Figure~\ref{fig:90} (species 37) by thick striping distinct from that of Figure~\ref{fig:91}.}
\label{fig:92}
\end{figure}

At level 4 (Figures~\ref{fig:87ab}a and b for $33_4$, Figures~\ref{fig:89ab}a and b for $35_4$, and Figures~\ref{fig:91} and \ref{fig:92} for 37), one needs to refer to the designs, Figures~\ref{fig:76ab}, \ref{fig:77ab}, and \ref{fig:90} respectively. In each of those designs, a level-4 $G_1$ lattice unit of the same size and orientation is delineated by solid lines. In all three cases -- and in the case of any level-4 lattice unit based on the same level-1 unit -- the blocks centred on the corners and centre of the lattice unit are spaced like a doubled $(5,3)$ satin. One way to colour fabrics of those species {\em based on the same level-1 unit} is to make those corner and centre blocks redundant. As at level 3, there are just as good alternative lattice units taking their corners to be the centres of the lattice units and centres to be the corners of the lattice units used. But these alternative lattice units are no help because the centres and corners have already been used. The parallel between levels 3 and 4 breaks down, and something different needs to be done. Fortunately, blocks centred on the mid-sides of the lattice units (both of the alternatives) are also spaced like a doubled $(5,3)$ satin -- because they are the corners of the inscribed level-3 lattice unit -- and are permuted the same way as the others. In each of the pairs of colourings illustrated above, the first has been done with the redundant blocks at corners and centre and the second at mid-sides of the level-4 lattice units like the unrelated Figures~\ref{fig:71ab} and \ref{fig:67abc}.

Because the previous two paragraphs pertain to lattice units at levels 3 and 4 and not to particular fabrics or even particular species, all of these subspecies can be coloured in two ways, and the theorem is proved.
\end{proof}

\begin{figure}
\[
\begin{array}{cc}
\epsffile{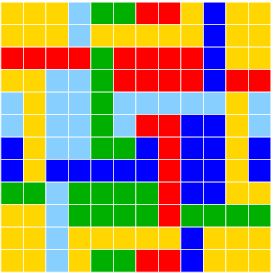} &\epsffile{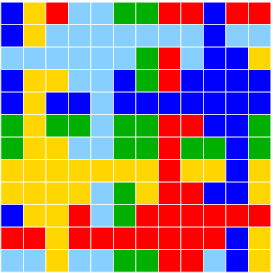}\\

\mbox{(a)} &\mbox{(b)}
\end{array}
\]
\caption{Distinct 5-colourings of 20-19437 by thick striping (species 38, Figure~\ref{fig:85a74b}a).}
\label{fig:86ab}
\end{figure}

\section{Level-1 lattice units larger than minimal}
\label{sect:larger}
\noindent Figure~\ref{fig:68abc}a discreetly introduces the fact that some isonemal fabrics with level-1 $p4$ lattice units having area larger than $5=M_1^2+N_1^2$ are equally suitable for 5-colouring with (5,3)-satin redundancy. It has $M_1=3, N_1=4$. The two requirements on $M_1$ and $N_1$ are that $G_2$ of a prefabric based on such a level-1 lattice unit be a subgroup of $G_2$ of the (5,3) satin (for 5-colorability) and that the prefabric be isonemal. The first requirement is easily visualized in the diagram of Figure~\ref{fig:79}. It has $x$- and $y$-axes 1 unit up and 3 units to the right of two diagram boundaries. If we put a point of the infinite (5,3)-satin lattice in its usual orientation at the origin, then its points are indicated by the boxes, some of which are filled.
\begin{figure}
\[
\epsffile{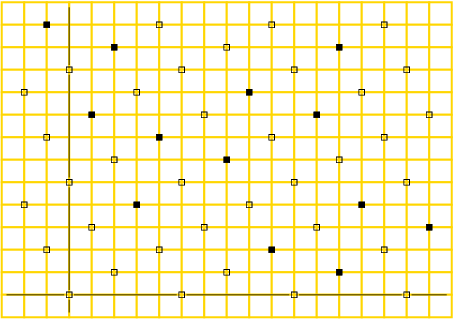}
\]
\caption{A quarter-plane of corners of $G_1$ lattice units of Figure \ref{fig:66abc}a of the $(5, 3)$ satin marked with boxes. $M_1 =2$. $N_1 = 1$. Corners eligible as corners of larger lattice units marked \blboxx , ineligible marked \whboxx .}
\label{fig:79}
\end{figure}

For 5-colorability with this redundancy, the level-1 lattice of a prefabric must fall within this lattice, as that of the fabric of Figure~\ref{fig:68abc}a does. The filled box closest to the origin, $(3,4)$, is the other end of an edge of its lattice unit starting at the origin. Going around the lattice unit, the next corner is at $(-1,7)$ (on the diagram) and the last at $(-4,3)$ (just off the diagram).

The isonemality constraint is much more complex, but it is a main topic of \cite{Thomas2010b}. Various points $(x, y)$ with integer co-ordinates are ineligible:

\noindent 1. Points with $x=5m$ or $y=5n$ (otherwise isonemality fails).

\noindent 2. Points with $(x, y) = m(2, 1)$ or $n(-1, 2)$ (otherwise isonemality fails or the level is greater than 1).

\noindent 3. Points with $x$ and $y$ not relatively prime (otherwise level is greater than 1).

\noindent 4. Points with $(x, y) = m(1, 3)$ (otherwise level is $2m$).

\noindent What one is looking for is $(x, y)$, representing $(M'_1,N'_1)$ such that $x^2 + y^2$ is the sum of an odd square and an even square. (For isonemality $M'_1$ and $N'_1$ must be of opposite parity and relatively prime.
In the figure, the closest eligible points to the origin in the first quadrant give lattice-unit areas $3^2 + 4^2=25$, $1^2 + 8^2=65=4^2 + 7^2$, $2^2 + 11^2=125$, $7^2 + 6^2=85=9^2 + 2^2$, $8^2 + 9^2=145=12^2+1^2$, $11^2 + 8^2=185=13^2 + 4^2$, and $12^2 + 11^2 = 265 = 16^2 + 3^2$.
Note the area duplications that are geometrically distinct in orientation.
Some ineligible pairs have  $7^2 + 1^2=50=2(25)$, $7^2+11^2=170=2(85)$, and $9^2+7^2=130=2(65)$, which are level-2 corners.
Because the configuration has quarter-turn symmetry about the origin, the other corners of squares on the displayed bases (or any such bases) are in the configuration.
Along each vertical half-line $x\neq 5m$ in the first quadrant beginning with $x=1$, $y > 0$, alternate boxed points have the correct (opposite) parity, differing by 10 in $y$-co-ordinates.
For example, $(4, 12)$ is no use, but $(4, 7)$ and $(4, 17)$ are fine.
The process of locating suitable corners is much the same as the determination of $M_1$ and $N_1$ in \cite{Thomas2010b}, where only isonemality and level 1 are issues, because what is wanted is the overlap between the result of that process and the $(5, 3)$-satin lattice.
\begin{figure}%
\[
\epsffile{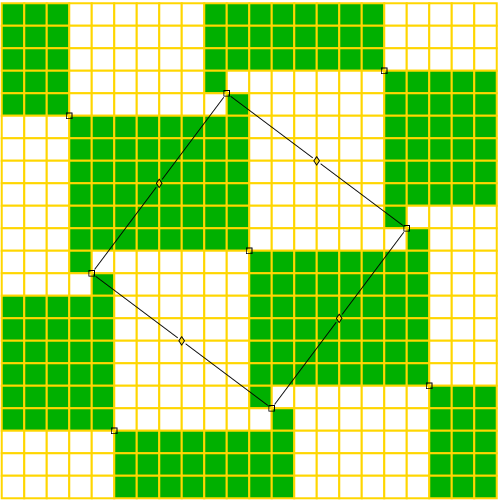}
\]
\caption{Order-50 example of species $33_3$ with one lattice unit of $G_1$ outlined and other centres of quarter turns marked.}
\label{fig:93}
\end{figure}

\begin{thm}
The groups $G'_2$ at levels 1 to 4 based on $M'_1=2m-n$, $N'_1=m+2n$, where $m$ and $n$ are integers of opposite parity, are subgroups of the group $G_2$ of the satin with $M_1=2$, $N_1=1$.
\label{thm:larger}
\end{thm}
\begin{proof}
What needs to be proved for level 1 is that, if the first side of a lattice unit is from the first corner (0,0) to $(M'_1,N'_1) = (2m-n,m+2n)$, then $(M'_1,N'_1)$ and the other two corners around conterclockwise, $(m-3n,3m+n)$ and $(-m-2n,2m-n)$, fall on corners, that the mid-sides fall at mid-sides, and that the centre falls at a centre of a $G_2$ lattice unit.
The second corner is obvious. It is necessary that $(m-3n,3m+n)$ and $(-m-2n,2m-n)$ be of the form $p(2,1) + q(-1,2)$. This is the case for $p=m-n$ and $q=m+n$ for the third corner and $p=-n$ and $q=m$ for the fourth.

The mid-points of the first and last sides of the potential lattice unit are $\textstyle{\left(\frac{2m-n}{2}, \frac{m+2n}{2}\right)}$ and $\textstyle{\left(-\frac{m+2n}{2}, \frac{2m-n}{2}\right)}$, with $m$ and $n$ of opposite parity.
Mid-points of the basic lattice units' edges are of the form A: $\textstyle{(1, \frac{1}{2}) + p(2, 1) + q(-1, 2)}$ and B: $\textstyle{(-\frac{1}{2}, 1) + p(2, 1) + q(-1, 2)}$ for $p$ and $q$ integers.
The first mid-point is of form A when $m$ is odd and $n$ is even and $\textstyle{p=\frac{m-1}{2}}$ and $\textstyle{q=\frac{n}{2}}$; it is of form B when $m$ is even and $n$ is odd and $\textstyle{p=\frac{m}{2}}$ and $\textstyle{q=\frac{n-1}{2}}$.
The last mid-point is of form A when $m$ is even and $n$ is odd and $\textstyle{p=-\frac{n+1}{2}}$ and $\textstyle{q=\frac{m}{2}}$; it is of form B when $m$ is odd and $n$ is even and $\textstyle{p=-\frac{n}{2}}$ and $\textstyle{q=\frac{m-1}{2}}$.
So those two mid-points fall at mid-sides of the basic $G_2$ lattice, but the others are spaced as the corners are spaced and so all fall there.
\begin{figure}
\[
\epsffile{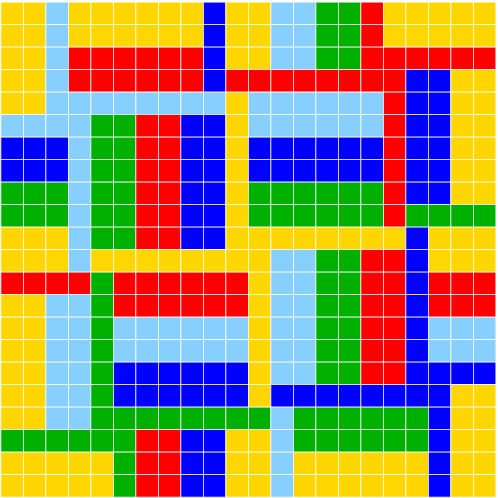}
\]
\caption{Five-colouring of fabric of Figure~\ref{fig:93} by thick striping.}
\label{fig:94}
\end{figure}

The centre is the mid-point $\textstyle{\left(\frac{m}{2}-\frac{3n}{2},\frac{3m}{2}+\frac{n}{2}\right)}$ of the diagonal vector to the third corner $(m-3n,3m+n)$.
As the centre of one basic lattice unit is $\textstyle{\left(\frac{1}{2},\frac{3}{2}\right)}$, the other centres are of the form $\textstyle{\left(\frac{1}{2},\frac{3}{2}\right)} + p(2, 1) +q(-1, 2)= (2p-q+\frac{1}{2},p+2q+\frac{3}{2})$ with $p$ and $q$ integers.
The condition required for the centre of the $G'_2$ lattice unit to fall on the centre of a basic $G_2$ lattice unit is that 
$$\left(\frac{m}{2}-\frac{3n}{2},\frac{3m}{2}+\frac{n}{2}\right)=\left(2p-q+\frac{1}{2},p+2q+\frac{3}{2}\right)$$
for $p$ and $q$ integers.
But this is equivalent to $\textstyle{p=\frac{m-n+1}{2}}$, $\textstyle{q=\frac{m-n-1}{2}}$, which are integers if and only if $m$ and $n$ are of opposite parity.
So the centres of the supposed subgroups' lattice units do fall on centres of the basic lattice units.
If $G'_2$ at level 1 is a subgroup of $G_2$, then groups at levels 2, 3, and 4 based on it are also subgroups of $G_2$, being subgroups of $G'_2$. The theorem is proved.
\end{proof}

\begin{figure}
\[
\epsffile{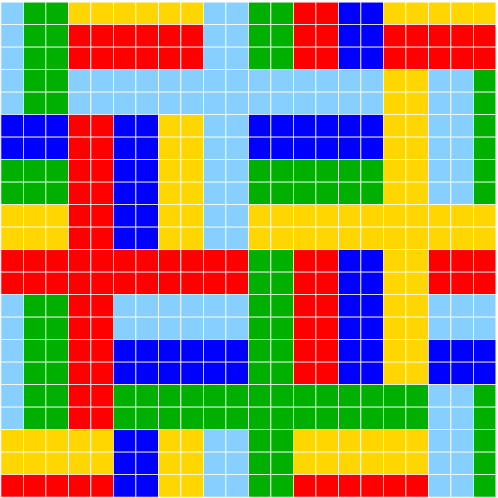}
\]
\caption{Five-colouring of fabric of Figure~\ref{fig:93} by thick striping, distinct from that of Figure~\ref{fig:94}.}
\label{fig:95}
\end{figure}

It is obvious that the lattice units with corners displayed in Figure~\ref{fig:79} are the first few of an infinite set with corners in the original lattice.
It almost looks as though there are too many such subgroups, as though practically any group will be a subgroup of $G_2$.
Not so.
The area of the lattice units here is $(2m-n)^2 + (m+2n)^2 =5m^2 + 5n^2$, always divisible by 5, the area of the basic lattice unit (recall the numerical examples before the theorem), a condition that is by no means universal as the partial table of allowable level-1 edge-lengths in \cite{Thomas2010b} shows.
The periods and so orders of all of these fabrics are multiples of 5.
The next integers one might use instead of are 13 and 17.
So 5 is the only number of colours to be used this way for thin striping of genuine interest.
The level-1 orders $\leq 100$ are only 5, 10, 20, 25, 50, $65=5(13)$, $85=5(17)$, and 100.
Because the isonemal fabrics thickly stripable with 5 colours are at levels 3 and 4, when they are based on the minimal level-1 lattice unit of area 5 they have orders 10 ($33_3$ and $35_3$) and 20 ($33_4$, $35_4$, 37, and 38).
When one looks farther afield for fabrics as for thin striping, one finds---besides the further ordinary level-1 orders 13, 17, and so on---that the same level-1 lattice units of area 25, 65, and 85 are available as for thin striping, but they can be used only at levels 3 and 4 where fabrics have orders 50, 100, 130, 170, 260, and 340 of steeply deminishing practical interest.
We conclude with an example of an order-50 fabric of subspecies $33_3$ in Figure~\ref{fig:93}, 5-coloured in Figures~\ref{fig:94} and \ref{fig:95}.
The second is again a pattern that looks doubled but cannot be, like Figure~\ref{fig:92}.

\section{Redundancy not square satin}
\label{sect:nonsatin}
\noindent Besides the other satins mentioned already, there are other isonemal non-twills with one dark cell per order length --- though only one with order less than 8. Not with order 3; the only isonemal fabric of order three is the twill \cite{GS1985}. Not with order 4; the highly exceptional fabric 4-1-1 cannot be used (\cite{Thomas2014}, proved in the proof of Theorem 3.3), and the only other such fabrics of order 4 are a twill and a doubled twill \cite{GS1985}. There is only one non-twill of order 5 \cite{GS1985}, the satin already discussed. Order 7 has only twills \cite{GS1985}. There is one isonemal fabric of order 6 with one dark cell per 6, 6-1-1 (called in \cite{GS1985} a {\em twillin}); it supplies the positive answer that is small enough to be visually interesting (Figure~\ref{fig:6-1-1}).
\begin{figure}
\[
\epsffile{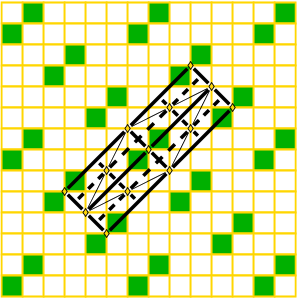}
\]
\caption{Example of twillin redundancy configuration, design 6-1-1 of species $27_o$.}
\label{fig:6-1-1}
\end{figure}

In order to appreciate how it can be used as a redundancy configuration, it is necessary to know more about isonemal fabrics that have only half-turn symmetries rather than quarter turns as well. Species 11 to 32 have half-turn symmetries on account of having axes of symmetry at right angles. They are the subject of \cite{Thomas2010a}. The point was made with Figures~\ref{fig:66abc}a and \ref{fig:67abc}a that there is no unique way to draw lattice units, and so there is a degree of arbitrariness to the lattice units that I use to illustrate these symmetry groups. This section takes up again the topic of Section~\ref{sect:sym}.

The simplest of these symmetry groups to appreciate is that illustrated in Figure~\ref{fig:nt2ab}a. One lattice unit of the symmetry group is outlined by mirrors (axes of reflective symmetry in the plane) represented by thick dark lines. They are dark to indicate the involvement of $\tau$. When cells are reflected in an axis of symmetry at $\pi/4$ to the horizontal, warps and wefts are interchanged, reversing dark and pale. But the symmetry here preserves colour, and so $\tau$ must be called upon to restore original colours. There are also a mirror through the middle of the lattice unit in each perpendicular direction and, where mirrors intersect, centres of half-turn symmetry indicated by \diaa . One can easily see the reflective and rotational symmetries in the figure. The mirrors and half turns repeat periodically all across the infinite plane. The boundaries in the diagram just indicate what the periods are. Perpendicular mirrors are characteristic of species 25 and 26 with their intersections differently placed in the different species. The example is of species 25.

Figure~\ref{fig:5cXX}a illustrates a lattice unit with no bounding lines; the axis representations just stop at the boundaries I chose. The axes here are of glide-reflections; they are hollow to indicate no involvement of $\tau$. No such involvement entails that dark cells are pale after their {\em side-preserving} glide-re\-flec\-tion and vice versa. The coherent dark and pale regions are interchanged by the glide-reflections in both directions. As one can easily see, the half turns entailed by the perpendicular glide-reflections are not at axis intersections but half-way between both sets of perpendicular axes. Species 11 to 16 have intersecting glide-reflection axes of the two kinds in various combinations, the example being of 11. 
Species 13 and 15 differ from 11 in having respectively side-reversing glide-reflection axes in both directions and perpendicular side-preserving and side-reversing glide-reflection axes. This makes the three identically colorable.
\begin{figure}
\[\begin{array}{cc}
\epsffile{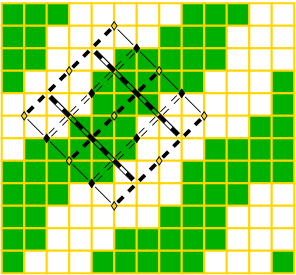} &\epsffile{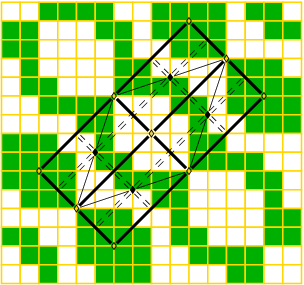}\\
\mbox{(a)} &\mbox{(b)}
\end{array}
\]
\caption{a. 8-7-2 of species 21 (\cite[Figure 10a]{Thomas2010a} and \cite[Figure 45a]{Thomas2014}). b. Roth's example 16-2499 of species 29 (\cite[Figure 15a]{Thomas2010a} and \cite[Figure 49]{Thomas2014}). These fabrics are 4-coloured in \cite[Figures 45b and 50a]{Thomas2014}.}
\label{fig:Perp}
\end{figure}

Figure~\ref{fig:5cXX}b illustrates a lattice unit with boundaries that are partly just thin boundary lines and partly axes of {\em side-reversing} glide-reflec\-tions involving $\tau$ to keep colours from being reversed. They pass down the middle of the pale and dark zig-zags. Between them are axes of side-preserving glide reflections that interchange dark and pale zig-zags. These are chosen boundaries; boundaries could just as well be axes of side-preserving glide-reflection instead. Perpendicular mirrors are easily noticed, but their solid black lines have dashed filling because they are coincident with axes of side-preserving glide reflections interchanging dark and pale zig-zags. Again, mirrors could have been used as lattice-unit boundaries. Because of the mixture of side-preserving and side-reversing glide-reflections, the half turns are some of them side-preserving, represented by \diaa , and some of them side-reversing (with $\tau$) represented by \diabb . The presence of glide-reflections perpendicular to mirrors is characteristic of species 17 to 24, the example being of 22.
Species $17_o$ and $19_o$ differ from 22 in having respectively only side-reversing glide-reflection axes perpendicular to the mirrors and only side-preserving glide-reflection axes perpendicular to the mirrors. This makes the three identically colorable.
Species 21 differs only in spacing from 22, which it otherwise resembles (example illustrated in Figure~\ref{fig:Perp}a). Whereas in 22 the distance parallel to mirrors between the closest \diaas is an even multiple of $\sqrt 2$ (to fit a \diab between them) and the distance perpendicular to mirrors between closest \diaas (and \diabs too) is odd, a similar description of 21 differs only in changing the last word to `even'. This matters. We abbreviate $\sqrt 2$ to $\delta$.

Figure~\ref{fig:13aXX}a illustrates a rhombic lattice unit bounded by thin lines because no axes are available for that purpose. Mirrors with their attendant centres of half turns surround and bisect the lattice unit, which is also cut by axes of side-preserving glide-reflection in both directions. In this environment, rather different from that of Figure~\ref{fig:5cXX}a (much more like that of Figure~\ref{fig:6-1-1}), side-reversing half-turn centres appear at the intersection of the glide-reflection axes. The presence of side-preserving glide-reflections parallel to mirrors in both directions is characteristic of species 29 to 32, the example being of 30.
More needs to be said of species 29 (example illustrated in Figure~\ref{fig:Perp}b). As the diagram shows, \diaas and \diabs are arranged in congruent lattices. It is characteristic of species 29 that the spacing of each lattice is even multiple of $\delta$ by even multiple of $\delta$. The lattice of \diabs on the glide-reflection axes has mirrors half-way between nearest pairs in both perpendicular directions. This spacing has consequences. Species $27_e$ has the same spacing although all the half-turn centres are side-preserving. 6-1-1 is of $27_o$ with these inter-\dia distances odd multiples of $\delta$.

Lastly, Figure~\ref{fig:6-1-1} illustrates the symmetry of the redundancy configuration 6-1-1 and the arrangement of species 27 and 28. It is similar to that of the previous paragraph except that its glide-reflection axes are side-reversing with side-preserving half-turn centres at their intersections. While species 28 is unusually interesting \cite{Thomas2010a}, it is not relevant to the current concern and will be ignored. 6-1-1 can obviously be 6-coloured with itself as redundancy configuration. This short discussion scratches the surface of the groups involved in these species but that is all. For a full discussion one needs \cite{Thomas2010a}.

\begin{thm}
The isonemal fabric 6-1-1 can be used as the redundancy configuration for perfect colouring with 6 colours by thin striping of some isonemal fabrics of species 11, 13, 15, 17, 19, 22, 25, 27, and 30.
\end{thm}

\begin{proof}
The proof, like that of previous theorems is by the existence of required colourings in species in the list and by argument that species not in the list cannot have examples. The argument comes first, and then non-exhaustive examples of the required colourings will be shown in the remainder of the section. Because the redundancy configuration 6-1-1 is not a twill, the order of the warp colours and the order of the weft colours differ throughout the examples.

Roth \cite{Roth1993} uses 6-1-1 as his example of his symmetry-group type 27 ($cmm/p2$, 11/-); it is of subspecies $27_o$ with reflection and glide-reflection axes alternating $\delta$ apart in one direction and $3\delta$ apart in the perpendicular direction and with centres of half turns at cell-corner intersections of each kind of axis. This is the symmetry group $G_1$, which has to be projected onto its isometry group $G_2$ for the purpose of considering which fabrics can be coloured with that redundancy configuration. Here $G_2$ consists of side-reversing mirrors, side-reversing glide-reflections, and the half turns. To be colourable with this redundancy, a fabric's isometry group $G^\prime_2$ must be a subgroup of the $G_2$ just described, the axes of which are so close together that there are many opportunities.

Obviously no quarter-turn operation can apply; this e\-li\-mi\-nates spe\-cies 33 to 39. Less obviously, the mirror position of the glide-reflection axes (i.e., diagonally across cells) eliminates as candidates all species having glide-reflection axes but not in mirror position ($1_e$, $1_o$, $2_e$, $2_o$, 4, $8_o$, 10, 12, 14, 16, $18_o$, $18_e$, 20, 24, $28_o$, and 32) as for all even numbers of colours with twilly redundancy \cite[Section 5]{Thomas2014}. 
 
\begin{figure}
\[\begin{array}{cc}
\epsffile{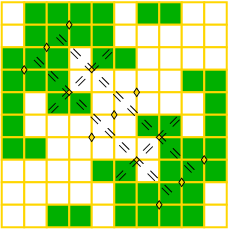} &\epsffile{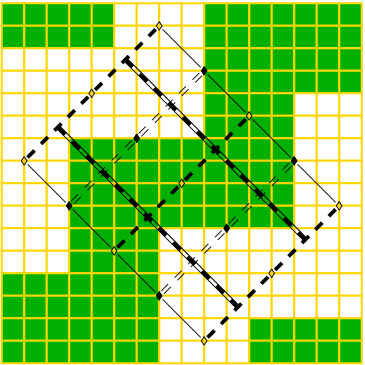}\\
\mbox{(a)} &\mbox{(b)}
\end{array}
\]
\caption{a. Design 12-111-2 of species $11_o$ (\cite[Figure 8a]{Thomas2013} and \cite[Figure 15a]{Thomas2014}). b. Species-22 example \cite[Figure 11a]{Thomas2010a}.}
\label{fig:5cXX}
\end{figure}

Species with only parallel axes other than $1_e$, $1_o$, $2_e$, $2_o$, 4, $8_o$, and 10 need to be considered. 
The axes of the fabric to be coloured must be spaced so as to match some of those of 6-1-1. The spacings of the axes (except for those of $1_e$, $1_o$, $2_e$, $2_o$, and 4) are conveniently set out in the proof of Theorem 2 in \cite{Thomas2009}. 

\noindent 1. In species $1_m$, $2_m$, and 3, there are only glide-reflection axes, and they are half an odd multiple of $\delta$ apart.

\noindent 2. In species $5_o$, the part of $5_e$ with $\ell$ even, 7, $8_o$, 10, there are mirrors, and they are half an odd multiple of $\delta$ apart.

\noindent 3. In the part of species $5_e$ with $\ell$ odd and species 6, there are mirrors a whole multiple of $\delta$ apart and with $G_1$ lattice units an odd multiple of $\delta$ long.

\noindent 4. In species $8_e$ and 9, there are mirrors an even number of $\delta$ apart with glide-reflection axes between them and with the glides an odd multiple of $\delta$.

Beyond classes 1 to 3 of spacings that are plainly unsuitable, number 4, with mirrors an even number of $\delta$ apart, likewise interleaved axes of glide-reflection with glides an odd multiple of $\delta$ long, looks promising. As the glides of 6-1-1 are $\delta$ in one direction and $3\delta$ in the other, species $8_e$ and 9 (differing only in whether glide-reflections are side-reversing ($8_e$) or side-preserving (9)) need to be considered, but the sizes of the rhombic lattice units needed for isonemality in species $8_e$ and 9, length $\ell\delta$ and width $w\delta$ with either $\ell$ and $w$ odd and $(\ell , w)=1$ (not possible for multiples of 2 and 6) or $\ell$ and $w$ even, $(\ell /2, w/2)=1$, and $\ell /2$ and $w/2$ different in parity (also impossible; see next paragraph) \cite[Lemma 4]{Thomas2009} cannot be reconciled with the sizes needed to match 6-1-1. No example can be constructed.

To see that $\ell$ and $w$ even and $\ell /2$ and $w/2$ different in parity is not possible, observe that the rhombic lattice units available for subgroups of 6-1-1's symmetry group must be a multiple of 6$\delta$ long and of 2$\delta$ wide. When $\ell$ is an odd multiple of 6, $w$ must be an odd multiple of 2. $\ell= 6(2m+1)$ and $w=2(2n+1)$ for some $m$ and $n$. Then $\ell/2= 3(2m+1)$ and $w/2=(2n+1)$, both odd. When $\ell$ is an even multiple of 6, $w$ must be an even multiple of 2. $\ell= 6(2m)$ and $w=2(2n)$ for some $m$ and $n$. Then $\ell/2= 6m$ and $w/2=2n$, both even. $\ell/2$ and $w/2$ always have the same parity.

Beyond the fabrics with no quarter-turn symmetry are species
11, 13, 15, 17, $18_s$, 19, 21--23, 25--27, $28_e$, $28_n$, and 29--31. Some centres of half turns in species $18_s$, 23, 26, $28_e$, $28_n$, and 31 are at cell centres and so cannot be imbedded in the configuration of centres of 6-1-1, which are all at cell corners. 

Two further species cannot have group $G_2$ embedded in that of 6-1-1, both for much the same reason. For species 21 there is no difficulty in placing \diaas and \diabs alternately $3\delta$ apart along a glide-reflection axis in 6-1-1 with positive slope in Figure~\ref{fig:6-1-1}, in each case with a mirror half-way between them. But the spacing of rows such as these is incompatible with the 6-1-1 configuration. These rows with positive slope must be an even multiple of $\delta$ apart and with a mirror half-way between them. But mirrors lie $\delta/2 + m\delta$ away from the glide-reflection axes with positive slope. Twice that distance is $(2m+1)\delta$. Only inter-axis distances odd in $\delta$ are divided in half by mirrors. This shows the incompatibility. In the perpendicular direction, mirrors lie $3\delta/2 + 3n\delta$ away. Twice that is $(6n+3)\delta$, again always odd; so rotation of easy rows through $\pi/2$ is blocked. The similar situation for species 29 is worse. Consider placing the lattice of \diabs in glide-reflection axes onto \diaas of 6-1-1 with mirrors half-way between them in both directions. The \diaas in axes in 6-1-1 are $\delta$ and $3\delta$ apart. If a \diab is placed on any \dia at intersecting axes, the adjacent \diab in neither perpendicular direction can be placed: mirrors are $\delta/2 + m\delta$ and $3\delta/2 + 3n\delta$ away. As for species 21, twice these distances are odd in $\delta$ but the next \diab must be an even multiple of $\delta$ away and so cannot be placed in either perpendicular direction.

Remaining are species 11, 13, 15, $17_o$, $19_o$, 22, 25, $27_o$, and 30, as required. Four examples from these will be discussed and illustrated.
\end{proof}
 
\begin{figure}
\[\begin{array}{cc}
\epsffile{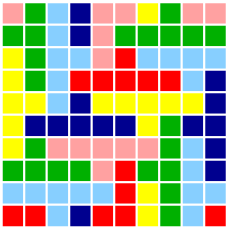} &\epsffile{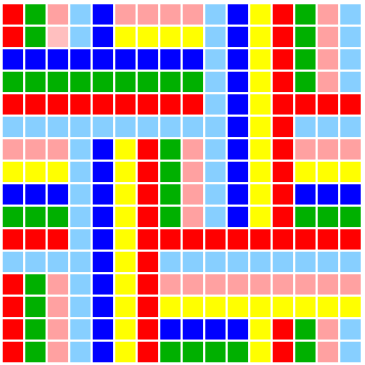}\\
\mbox{(a)} &\mbox{(b)}\\
\epsffile{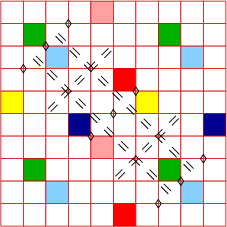} &\epsffile{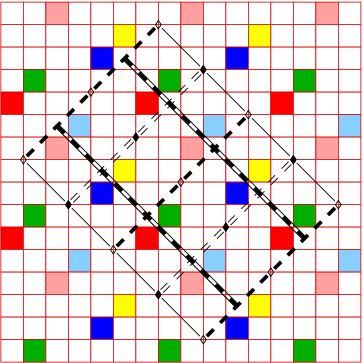}\\
\mbox{(c)} &\mbox{(d)}
\end{array}
\]
\caption{a and b. Six-colourings of Figure~\ref{fig:5cXX}a and b. c and d. Coloured redundant cells in a and b with symmetries shown from Figure~\ref{fig:5cXX}.}
\label{fig:5a3b}
\end{figure}

Roth's \cite{Roth1993} example (Figure~\ref{fig:5cXX}a) of species $11_o$, 12-111-2, was 3-coloured thickly in \cite[Figures 8b and c]{Thomas2013} and thinly in \cite[Figure 16a]{Thomas2014}. Its glide-reflection axes and the positions of its half-turn centres can be made to correspond to those of 6-1-1. Accordingly, its group $G_2$ is a subgroup of that of 6-1-1, and it can be 6-coloured as in Figure~\ref{fig:5a3b}a.
Since the glide-reflections are side-preserving, one sees in the design all the symmetry that there is.
When one looks at Figure~\ref{fig:5a3b}a, one does not see the redundancy configuration as easily as in the colourings of designs with quarter turns. For this reason and because one can check that the colouring is perfect by observing just how the configuration of redundant cells is transformed by the design's symmetry group (really just $G_2$), these colourings are supplemented by diagrams showing redundant cells -- as in Figure~\ref{fig:75b}.

One sees in both Figures~\ref{fig:5a3b}c and d what also appears in Figures~\ref{fig:nt2ab}c and \ref{fig:13aXX}c, a square lattice in each colour. In all four cases the pictures are the same, just oriented and coloured variously. Figure~\ref{fig:5a3b}d has the advantage of showing at least two lattice squares of each colour, and so I shall discuss it, but because the lattices are the same any remark applies equally to the others. The operations that are symmetries of the fabric design, because of the subgroup relation between the $G_2$ groups, are operations that leave the redundancy configuration as a whole invariant. Because they are all invertible rigid motions, the lattices remain square lattices related in the same way. They cannot be piled on top of one another or mixed. All that can happen to them is colour permutation as wholes. All of the operations involved perform coherent colour permutations of the lattices, as required. 
For example, in Figure~\ref{fig:5a3b}d, the reflection that interchanges the red cell and the green cell in the middle of the diagram interchanges the whole of the red lattice with the whole of the green lattice, the pink and the pale blue, and the yellow and dark blue. To check that each of the symmetries has that kind of effect on the redundant cells, as in Figure~\ref{fig:75b}, is to check that the symmetry has the required coherent effect. To check that they all do is to check that the colouring is perfect. There is no need to check the colouring as a whole, as in Figure~\ref{fig:5a3b}b; it is there to show what the pattern looks like. Species $13_o$ and $15_o$ can be coloured like $11_o$.

The example of Figure~\ref{fig:5cXX}b of order 48 was made up for \cite{Thomas2010a} to display species 22 and appears in \cite{Thomas2014}, 3-coloured in Figure~14b, 4-coloured in Figure~42a, and 6-coloured in Figure~27, all with twilly redundancy. A pair of dark cells of 6-1-1 can be laid down in cells having positions $(2,6)$ and $(3,5)$ in matrix notation and so on, that is, adjacent to axes of glide-reflection of positive slope and straddling the orthogonal mirrors. Fabric symmetries then leave them collectively invariant. This allows the 6-colouring in Figure~\ref{fig:5a3b}b. Observe that the mirrors leaving the design invariant are combined with $\tau$, so that their effect on the coloured fabric is to put motifs from each side onto the other side. Reflective symmetry is invisible in the colouring, but the effect can be seen in the redundant cells. Species $17_o$ and $19_o$ can be coloured like 22.

\begin{figure}
\begin{center}
\epsffile{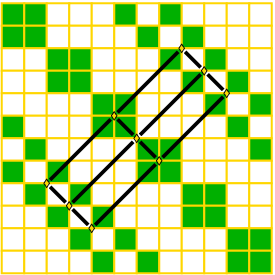}

(a)
\end{center}
\[\begin{array}{cc}
\epsffile{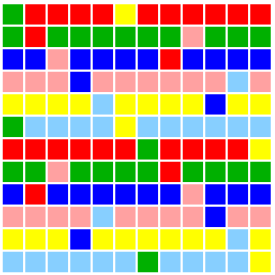} &\epsffile{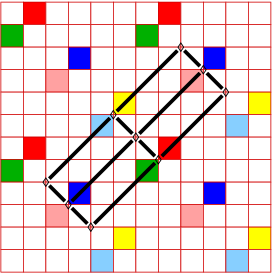}\\
\mbox{(b)} &\mbox{(c)}
\end{array}
\]
\caption{a. Design 12-163-1 of species $25_o$. b. 6-colouring of a. c. Coloured redundant cells in b with symmetries shown from a.}
\label{fig:nt2ab}
\end{figure}

Because of the large scale of the design of Figure~\ref{fig:5cXX}b, it is not easy to see what is the overall pattern of Figure~\ref{fig:5a3b}b. A verbal description, arrived at looking at a larger portion of the pattern, may help. From lower-left to upper-right there is a barbed red zig-zag. Eight cells diagonally up and to the left there is a congruent zig-zag in dark blue, and down and to the right one in pink (2 cells showing at bottom right). Moving the same way there are three zig-zags with barbs facing oppositely because they are glide-reflection images of the first three; they are in green (2 cells showing at top left), light blue, and yellow. Keeping these from touching except at endpoints are rectangles consisting of two rows of cells four long and ten long in contrasting colours. In the middle of the diagram a short green and pink rectangle separates the light-blue zig-zag from the red one. Two more of these same rectangles appear near the top right and bottom left corners, six cells diagonally up and to the right/down and to the left. Also conspicuous are the long blue and yellow rectangles, from top and bottom edges, separating the red and light-blue zig-zags. The side-reversing glide reflections and half turns produce their effects onto the opposite side like mirrors (with $\tau$). The side-preserving glide-reflections and half turns on the other hand have visible effects that include colour permutations.

Roth's \cite{Roth1993} example (Figure~\ref{fig:nt2ab}a) of species $25_o$, 12-163-1, can be 6-coloured as in Figure~\ref{fig:nt2ab}b because its mirrors can be made to coincide with those of 6-1-1. 
This pattern has the feature that in \cite{Thomas2011} I called `stripiness' not because there is anything aesthetically negative about stripes but because there is a predominance of stripes in a single direction in contrast with the balance of vertical and horizontal stripes in Figure~\ref{fig:5a3b}b or \ref{fig:13aXX}b. The balance there results from the same size and shape of the blocks of pale and dark in its design, and the imbalance here results from the predominance in this design of pale cells. The effect is reversed on the reverse, where the stripiness is the same but the direction orthogonal on account of the side-reversing effect of the mirrors. 
\begin{figure}
\begin{center}
\epsffile{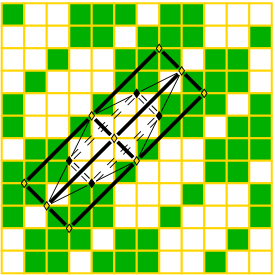}

(a)
\end{center}
\[\begin{array}{cc}
\epsffile{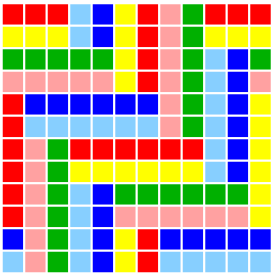} &\epsffile{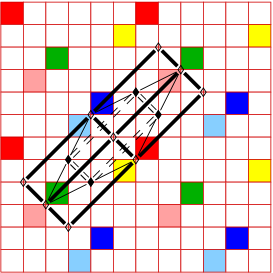}\\ 
\mbox{(b)} &\mbox{(c)}
\end{array}
\]
\caption{a. Design 12-315-4 of species 30. b. 6-colouring of a. c. Coloured redundant cells in b with symmetries shown from a.}
\label{fig:13aXX}
\end{figure}

Roth's \cite{Roth1993} example (Figure~\ref{fig:13aXX}a) of species 30, 12-315-4, is an interesting example of how the manner of colouring makes a difference to the result. This example was used in \cite{Thomas2013} for thick 3-colouring to produce a pattern (Fig.~9b there) with interesting motifs and in \cite{Thomas2014} for thin 3-colouring to produce (Figure~7b there) motifs arranged interestingly. If it is 6-coloured with the dark cells of 6-1-1 {\em between} the diagonally adjacent pairs of dark cells in the design, that is, with the illustrated mirrors and axes of 6-1-1 corresponding to those of 12-315-4, the result is bits of colour no larger than three in a row. On the other hand, if the dark cells of 6-1-1 are put {\em on} the diagonally adjacent pairs of dark cells in the design, the result (Figure~\ref{fig:13aXX}b) is longer stripes arranged as pleasingly as six colours can be arranged -- not quite herringbone -- and in different orders vertically and horizontally on account of the non-twilly redundancy.

What is shown is actually that the list of allowed species could be more narrowly stated as $11_o$, $13_o$, $15_o$, $17_o$, $19_o$, 22, $25_o$, $27_o$, and 30.

\section{Weaving non-planar surfaces}
\label{sect:nonplanar}
\noindent Weaving and colouring flat tori were discussed in the papers on twilly redundancy. The opposite edges of a period parallelogram of a design can be identified to weave a torus, and the opposite edges of a period parallelogram of a pattern (often not minimal) can be identified to weave a torus in colour. Since the colours in a pattern may not match up in a period parallelogram of a design, the parallelogram for edge identification of a pattern may have to be larger, but there is never any difficulty in taking enough small parallelograms to make the colouring come out even. For example in Figure~\ref{fig:81b72b} of 10-39-1, either the whole left diagram can be used (with ten strands in each direction) to weave a torus or the whole right diagram to 5-colour it. The oblique square outlined solidly in Figure~\ref{fig:81b72b}a (a symmetry-group lattice unit) can be used to weave a torus, but it is much too small to use to colour one this way. Red and yellow cells along its lower right boundary would have to be identified with light blue and red cells respectively along its upper left boundary. It requires a square made up of 25 such lattice units (5 by 5) for the colours to come out even. An interesting coincidence occurs in this example despite the sizes of the tori being so different (100 cells vs 500 cells). In both cases there are two strands of each colour crossing the torus in each perpendicular direction. The oblique strands are much longer, because each of them crosses the larger square twice before its ends are identified. So nothing important for tori is changed with non-twilly redundancy.
\begin{figure}
\[
\epsffile{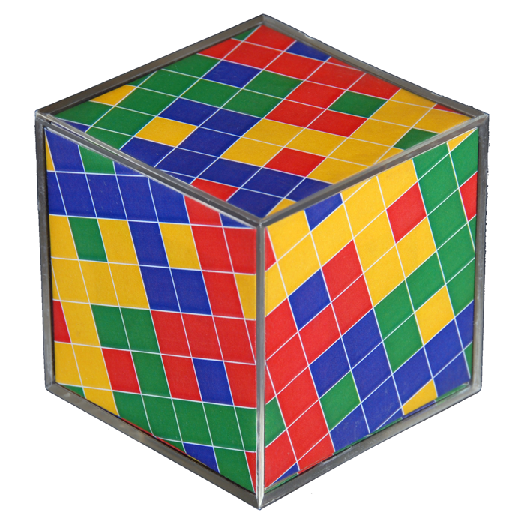}
\]
\caption{Downward view of 4-coloured cube of \cite[Figures 14 and 15]{Thomas2012}.}
\label{fig:colCubeFront}
\end{figure}

That is quite different from weaving cubes, where redundancy invariant under quarter turns is essential if anything like isonemality is to occur. This idea began with the suggestion from G.C.~Shephard to Jean Pederson, published in \cite{Pedersen1981}, of a net for a cube divided up like the lower lattice unit of Figure~\ref{fig:67abc}a to show that it was possible to weave a cube in that way. A full net for a different cube appears as Figure 16a of \cite{Thomas2010b}. Weaving cubes was discussed in \cite{Thomas2010b} and 2-colouring them with 2-fold, 3-fold, and 4-fold rotational colour symmetries about edge centres, vertices, and face centres respectively was discussed in \cite{Thomas2012}. To conclude this paper without spending a lot of time and space on the matter, I display in Figures~\ref{fig:colCubeFront} and \ref{fig:colCubeBack} an example of a perfect 4-colouring of the example in \cite{Thomas2012} also preserving the rotational symmetries of the cube.

It is of some interest that the underlying fabric of this 4-coloured cube (Figure~\ref{fig:98}, an expansion of Figure 13 of \cite{Thomas2012}) is 5-colourable not 4-colourable, being of species $33_4$.
Figure~\ref{fig:98} is a diagram of the design of the fabric extended to contain like a net the three faces of the cube visible in Figure~\ref{fig:colCubeFront}. The lower left edge of the net with positive slope corresponds to the lower right edge of the cube.
The conventions for illustrating thick striping with boxes outside a design do not quite fit what needs to be done here; one vertical stripe is marked as half blue and half red because the strand that is its upper part has to be blue and the different strand that is its lower part has to be red.
The colours of the vertical strands in the lower right lattice unit are merely the continuation of the relevant colours of the horizontal strands in the upper lattice unit since the strands are continuations in the cube.
In a 5-colouring of the fabric, only four colours feature in each lattice unit, and using only those four colours for the net obviously does not extend to the whole plane.
\begin{figure}
\[
\epsffile{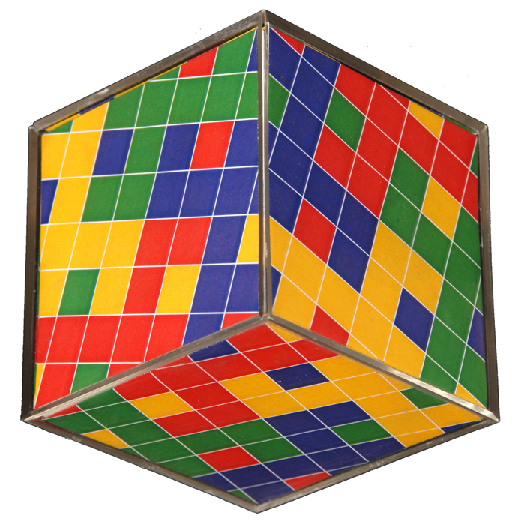}%
\]
\caption{Upward view of the back of the cube of Figure~\ref{fig:colCubeFront} as reflected in a mirror.}
\label{fig:colCubeBack}
\end{figure}

My attention was drawn to the weaving of polyhedra by a later paper of Jean Pedersen \cite{Pedersen1983} in \emph{The Mathematical Intelligencer}. \emph{The Mathematical Intelligencer} has struck again. In a single issue it had a posthumous paper by Paulus Gerdes \cite{gerdes2015} discussing the weaving of polyhedra in African cultures and one \cite{dunning2015} discussing the topological models of Alexander Crum Brown. The former suggests mentioning the distinction between weaving a polyhedron as above, where the faces are woven, and plaiting it with a strip (or with more than one strip) so that each whole face is composed of a single triangle or quadrilateral from that strip, {\em each face not being woven}. Gerdes illustrates both. Crum Brown used colour to make it easy to distinguish the different surfaces that he interlaced by making each one a different colour. That is always possible, sometimes helpful, and could be used on the cubes above by making the eight strands be of eight colours rather than striping thickly with four.
Jean Pedersen too has struck again \cite{HP2010} with more on weaving polyhedra including cubes (Ch.~9) and a discussion of cubes' symmetry (Ch.~10) but with cubes less interesting than the above.
\begin{figure}
\[
\epsffile{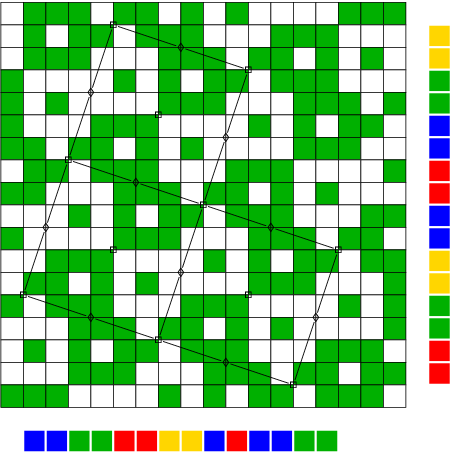}
\]
\caption{Design used for cubes in Figures 14 and 15 of \cite{Thomas2012} with colouring indication for Figures~\ref{fig:colCubeFront} and \ref{fig:colCubeBack}. Unconventional colouring indication is explained in text.}
\label{fig:98}
\end{figure}

\nocite{*}
\bibliographystyle{amsplain}

\end{document}